\documentclass{amsart}
\usepackage{amssymb}
\usepackage{mathrsfs}
\usepackage{stmaryrd}
\usepackage{bbm}
\usepackage{oldgerm}
\usepackage[english]{babel}
\usepackage[T1]{fontenc}
\usepackage[latin1]{inputenc}
\usepackage[all]{xy}
\usepackage{mathabx}
\usepackage{graphicx}
\usepackage[colorlinks,linkcolor=red,anchorcolor=red,citecolor=blue]{hyperref}

\newtheorem{theorem}[subsection]{Theorem}
\newtheorem{proposition}[subsection]{Proposition}
\newtheorem{corollary}[subsection]{Corollary}
\newtheorem{lemma}[subsection]{Lemma}

\theoremstyle{definition}

\newtheorem{remark}[subsection]{Remark}

\newtheorem{example}[subsection]{Example}
\numberwithin{equation}{subsection}

\newcommand{\be}{\begin{equation}}
\newcommand{\ee}{\end{equation}}

\newcommand{\ct}{{\mathbb T^\ast}}

\newcommand{\rR}{\mathrm R}
\newcommand{\rP}{\mathrm P}
\newcommand{\rpr}{\mathrm{pr}}

\newcommand{\bT}{\mathbb T}
\newcommand{\bh}{{\mathbb H}}
\newcommand{\bH}{\mathbb H}
\newcommand{\bP}{\mathbb P}
\newcommand{\bp}{{\mathbb P}}
\newcommand{\bA}{\mathbb A}
\newcommand{\cf}{{\mathcal F}}
\newcommand{\cF}{\mathcal F}
\newcommand{\cG}{\mathcal G}
\newcommand{\cK}{\mathcal K}
\newcommand{\cC}{\mathcal C}
\newcommand{\cO}{\mathcal O}

\newcommand{\red}{\mathrm{red}}
\newcommand{\supp}{\mathrm{supp}}

\newcommand{\bZ}{\mathbb Z}
\newcommand{\bV}{\mathbb V}
\newcommand{\bG}{\mathbb G}
\newcommand{\bD}{\mathbb D}
\newcommand{\wt}{\widetilde}

\newcommand{\dimtot}{\mathrm{dimtot}}

\begin{document}

\title[Relative singular support and the semi-continuity of characteristic cycles]{Relative singular support and the semi-continuity of characteristic cycles for \'etale sheaves}

\author{Haoyu Hu}
\address{Graduate School of Mathematical Sciences, the University of Tokyo, 3-8-1 Komaba Meguro-ku Tokyo 153-8914, Japan}
\email{huhaoyu@ms.u-tokyo.ac.jp, huhaoyu1987@163.com}

\author{Enlin Yang}
\address{Fakult\"at f\"ur Mathematik, Universit\"at Regensburg, 93040 Regensburg, Germany}
\email{ enlin.yang@mathematik.uni-regensburg.de, yangenlin0727@126.com}



\subjclass[2000]{Primary 14F05; Secondary 14C25}



\keywords{Semi-continuity, relative singular support, relative characteristic cycle}

\begin{abstract}
Recently, the singular support and the characteristic cycle of an \'etale sheaf on a smooth variety over a perfect field are constructed by Beilinson and Saito, respectively. In this article, we extend the singular support to a relative situation. As an application, we prove the generic constancy for singular supports and characteristic cycles of \'etale sheaves on a smooth fibration. Meanwhile, we show the failure of the lower semi-continuity of characteristic cycles in a higher relative dimension case, which is different from Deligne and Laumon's result in the relative curve case.

\end{abstract}
\maketitle

\tableofcontents

\section{Introduction}

\subsection{}
Let $Y$ be a smooth complex algebraic variety of equidimension $d$, $\mathcal D_Y$ the sheaf of differential operators on $Y$ and $\cG$ a holonomic left $\mathcal D_Y$-module. We  can associate two geometric invariants to $\cG$. One is the singular support $SS(\mathbb \cG)$, which is a closed subset of the cotangent bundle $\bT^*Y$ of equidimension $d$. The other one is the characteristic cycle $CC(\cG)$, which is a cycle of $\bT^*Y$ supported on $SS(\mathbb \cG)$.
 They contain important information on the $\mathcal D_Y$-module $\cG$. For example, assuming that $Y$ is projective over $\mathbb C$, we have \cite{Dubson,ka83}
 \begin{equation}\label{dmodindex}
 \chi(Y,\mathrm{DR}_Y(\cG))=\mathrm{deg}(\bT^*_YY,CC(\cG))_{\bT^*Y},
 \end{equation}
where $\chi(Y,\mathrm{DR}_Y(\cG))$ denotes the Euler-Poincar\'e characteristic of the de Rham complex $\mathrm{DR}_Y(\cG)=\Omega^1_{Y/\mathbb C}\otimes^L_{{\mathcal D}_Y}\cG$  and $\bT^*_YY$ the zero-section of $\bT^*Y$.

\subsection{}\label{ladiccc}
It is well known that $\ell$-adic \'etale sheaves are an analogue of algebraic $\mathcal D$-modules. Let $X$ be a smooth scheme over a field $k$, $\ell$ a prime number invertible in $k$, $\Lambda$ a finite field of characteristic $\ell$ and $\cK$ a bounded complex of constructible sheaves of $\Lambda$-modules on $X$.
Recently, an important breakthrough is achieved by Beilinson. Using a Radon transform method initiated by Brylinski, he defined the singular support $SS(\cK)$ of $\cK$ and proved the existence of $SS(\cK)$ \cite{Beilinson15}. It is a closed conical subset of the cotangent bundle $\bT^*X$. When $X$ is equidimensional, he proved that every irreducible component of $SS(\cK)$ has the same dimension as $X$ ({\it loc. cit.}).

Assume that $k$ is perfect. Based on Beilinsion's results, Saito constructed the characteristic cycle $CC(\cK)$ of $\cK$ as an algebraic cycle of $\bT^*X$ supported on $SS(\cK)$ \cite{Saito15-2}. The coefficient of each component of $CC(\cK)$ is obtain by the total dimension of the stalk of a vanishing cycles complex associated to $\cK$.
The cycle $CC(\cK)$ is positive if $\cK$ is perverse.
When $X$ is projective over $k$, he proved the following index formula ({\it loc. cit.})
\begin{equation}\label{introindex}
\chi(X_{\bar k},\cK|_{X_{\bar k}})=\deg(\bT^*_XX,CC(\cK))_{\bT^*X},
\end{equation}
where $\bar k$ denotes an algebraic closure of $k$ and $\bT^*_XX$ the zero-section of $\bT^*X$.
It is similar to the one for $\mathcal D$-modules \eqref{dmodindex}.

Assume that $X$ is a connected smooth curve over $k$ and that $\cK$ is the extension by zero of a locally constant and constructible sheaf of $\Lambda$-modules on an open dense subset $U$ of $X$. Then, we have
\begin{align}
CC(\cK)&=-\mathrm{rk}_{{\Lambda}}(\cK|_U)\cdot[\bT^*_XX]-\sum_{x\in X-U}\mathrm{dimtot}_x(\cK)\cdot[\bT^*_xX],
\end{align}
where $\mathrm{dimtot}_x(\cK)$ denotes the total dimension of $\cK$ at $x$ (\ref{defdimtot}) and $\bT^*_xX$ the fiber $\bT^*X\times_Xx$. The singular support $SS(\cK)$ is the reduced closed subscheme of $\bT^*X$ where $CC(\cK)$ is supported.  The index formula \eqref{introindex} is the well-known Grothendieck-Ogg-Shafarevich formula \cite{sga5}.

\subsection{}\label{introsetting}
In the rest of the introduction, let $S$ be an excellent Noetherian scheme, $f:X\rightarrow S$ a separated and smooth morphism of purely relative dimension $d$, $D$ an effective Cartier divisor of $X$ relative to $S$, $U$ the complement of $D$ in $X$ and $j:U\rightarrow X$ the canonical injection. For any $s\in S$, we denote by $\bar s$ an algebraic geometric point above $s$ and by $X_{\bar s}$ and $D_{\bar s}$ the fibers of $f:X\rightarrow S$ and $f|_D:D\rightarrow S$ at $\bar s$. We denote by $\bT^*(X/S)$ the relative cotangent bundle on $X$ with respect to $S$  and by $\bT^*_X(X/S)$ the zero-section of $\bT^*(X/S)$ (\ref{relcotbundle}).
Let $\ell$ be a prime number invertible in $S$, $\Lambda$ a finite field of characteristic $\ell$, $\cG$ a locally constant and constructible sheaf of $\Lambda$-modules on $U$ of constant rank and $\cF=j_!\cG[d]$.

\begin{theorem}[{\cite[2.1.1, 5.1.1]{Laumon81}}]\label{introdellau}
Assume that $f:X\rightarrow S$ has relative dimension $d=1$. In this case, $f|_D:D\rightarrow S$ is quasi-finite and flat. Then, 
\begin{itemize}
\item[(i)]
There exists a closed subscheme $A$ of $\bT^*(X/S)$ which is open and of purely relative dimension $d$ over $S$ such that, for any algebraic geometric point $\bar s$ of $S$, we have $SS(\cF|_{X_{\bar s}})=(A_{\bar s})_{\mathrm{red}}$. (In fact, we take $A=(\bT^*(X/S)\times_XD)\bigcup \bT^*_X(X/S)$).
\item[(ii)]
There exists a finite and surjective morphism $\pi:S'\rightarrow S$ and a cycle $B$ of $\bT^*((X\times_SS')/S')$ supported on  $A\times_SS'$ such that there exists an open dense subset $W'$ of $S'$ and, for any algebraic geometric point $\bar s'$ of $W'$, we have
$CC(\cF|_{\bar s'})=B_{\bar s'}$ (\ref{notationfiber});
\item[(iii)]
For any algebraic geometric point $\bar s'$ of $S'-W'$, the cycle $B_{\bar s'}-CC(\cF|_{\bar s'})$ on $\bT^*X_{\bar s'}$ is effective;
\item[(iv)]
Let $V$ be an open dense subset of $S$. Then, $f_V:X\times_SV\rightarrow V$ is locally acyclic with respect to $\cF|_{X\times_SV}$ if and only if, for any algebraic geometric point $\bar s'$ of $V\times_SS'$ , we have $CC(\cF|_{\bar s'})=B_{\bar s'}$.
\end{itemize}
\end{theorem}

Theorem \ref{introdellau} is the semi-continuity property for the Swan conductors of $\ell$-adic sheaves on relative curves due to Deligne and Laumon. Here, we formulate it in terms of characteristic cycles of $\ell$-adic sheaves. Saito asked whether the theorem can be generalized to the case where the relative dimension $d\geq 2$. In this article, we answered his question for part (i)-(iii) of Theorem \ref{introdellau}.  


\subsection{}
The main construction of this article is a relative version of singular supports of $\ell$-adic sheaves which generalizes Beilinson's one in \cite{Beilinson15}. Let $Y$ and $Z$ be two schemes smooth over $S$, $g:Y\rightarrow X$ and $h:Y\rightarrow Z$ two $S$-morphisms, and  $dg:\bT^*(X/S)\times_XY\rightarrow\bT^*(Y/S)$ and $dh:\bT^*(Z/S)\times_ZY\rightarrow\bT^*(Y/S)$ two maps induced by canonical morphisms $g^*\Omega^1_{X/S}\rightarrow \Omega^1_{Y/S}$ and $h^*\Omega^1_{Z/S}\rightarrow \Omega^1_{Y/S}$, respectively. For a closed conical subset $C$ of $\bT^*(X/S)$, we say that $g$ is $C$-{\it transversal relative to} $S$ if $(dg^{-1}(\bT^*_Y(Y/S)))\cap (C\times_XY)$ is contained in the zero-section of $\bT^*(X/S)\times_XY$. We denote by $g^{\circ}C$ the scheme theoretic image $dg(C\times_XY)$ in $\bT^*(Y/S)$. For a closed conical subset $C'$ of $\bT^*(Y/S)$, we say that $h$ is $C'$-{\it transversal relative to} $S$ if $C'\cap \mathrm{im}(dh)$ is contained in the zero-section of $\bT^*(Y/S)$.
Let $\cK$ be a bounded complex of constructible sheaves of $\Lambda$-modules on $X$.
We say that $\cK$ is {\it micro-supported} on a closed conical subset $C$ of $\bT^*(X/S)$ if, for any pair of morphisms $(h,g):Z\leftarrow Y\rightarrow X$ of schemes smooth over $S$ such that $g$ is $C$-transversal relative to $S$ and that $h$ is $g^{\circ}C$-transversal relative to $S$, the morphism $h$ is locally acyclic relative to $g^*\cK$. If there is a smallest element among closed conical subsets of $\bT^*(X/S)$ on which $\cK$ is micro-supported, we call it the {\it singular support of} $\cK$ {\it relative to} $S$ and we denote it by $SS(\cF,X/S)$ (cf. \ref{sub:rss}). 

The definition above was firstly introduced by Beilinson when $S$ is spectrum of a field \cite{Beilinson15}. Unlike the existence of singular support in {\it loc. cit.}, the relative singular support $SS(\cF, X/S)$ does not exist in general. However, using a Radon transform method similar to Beilinson's, we obtain the following theorem:

\begin{theorem}[{Theorem \ref{existrss}}]
The relative singular support $SS(\cK, X/S)$ exists after replacing $S$ by a Zariski open dense subscheme.
\end{theorem}

Let $s$ be a point of $S$. The similarity of definitions of the singular support and its relative version allows us to compare $SS(\cK,X/S)\times_Ss$ with $SS(\cK|_{X_s})$. We prove the following theorem. 

\begin{theorem}[Theorem \ref{sseqrssfiber} and Theorem \ref{closedA}]\label{introclosedA}
After replacing $S$ by a Zariski open dense subset, the relative singular support $SS(\cK, X/S)$ is flat over $S_{\mathrm{red}}$. Moreover, for any point $s$ of $S$, we have 
\begin{equation}
(SS(\cK,X/S)\times_Ss)_{\mathrm{red}}=SS(\cK|_{X_s}).
\end{equation}
\end{theorem}
 Theorem \ref{introclosedA} implies that, in higher relative dimensional cases, we can take $A=SS(\cK,X/S)$ to generalize part (i) of Theorem \ref{introdellau}, if allowing to shrink $S$. However, Example \ref{ssnotconst} implies that we cannot generalize part (i) in general. 
 
 Part (ii) of Theorem \ref{introdellau} is generalized by the following theorem:

\begin{theorem}[{Theorem \ref{maintheocc}}]\label{intromaintheocc}
There exists a dominant and quasi-finite morphism $\pi:S'\rightarrow S$ and a cycle $B$ on $\mathbb T^*((X\times_SS')/S')$ such that,
\begin{itemize}
  \item[(i)]
  $B=\sum_{\alpha\in I}m_\alpha[B_\alpha]$ $(m_\alpha\in\mathbb Q)$, where $B_\alpha$ ($\alpha\in I$) is flat over $S'$;
  \item[(ii)]
  For any algebraic geometric point $\bar t\rightarrow S'$, we have
\begin{equation*}
 B_{\bar t}=CC(\cK|_{X_{\bar t}}).
\end{equation*}
  \end{itemize}
\end{theorem}

\subsection{}

To prove Theorem \ref{intromaintheocc}, we may assume that $S$ is integral. We denote by $\eta$ the generic point of $S$ and $\bar \eta$ an algebraic geometric point above $\eta$. The characteristic cycle $CC(\cK|_{X_{\bar\eta}})$ on $\bT^*X_{\bar\eta}$ can be descent to a cycle $C$ on $\bT^*X_{\eta'}$, where $\eta'$ is a finite extension of $\eta$. We denote by $S'$ the integral closure of $S$ in $\eta'$ and put $X'=X\times_SS'$. Let $B$ be a cycle of $\bT^*(X'/S')$ such that $Y_{\eta'}=C$. Using the semi-continuity of the total dimensions of vanishing cycles complexes due to Saito \cite{Saito15-2}, we proved that coefficients of characteristic cycles of $\cK$ along fibers of $f':X'\rightarrow S'$ are locally constant. Replacing $S'$ by a Zariski open dense subscheme, we obtain Theorem \ref{intromaintheocc}. 
In response to the notion of relative singular support, we call $B$ the {\it characteristic cycle} of $\cK|_{X'}$ {\it relative to} $S'$.

\subsection{}
 We show that part (iii) of Theorem \ref{introdellau} cannot be generalized to a higher relative dimension situation by Example \ref{ccnotconst}. In other words, it says that the characteristic cycles is no longer lower semi-continuous in a higher relative dimension case, which is unlike the relative curve situation (part (2) of Theorem \ref{introdellau}). The main reason of its failure is that,  on a higher dimensional scheme, the singular support of an $\ell$-adic sheaf with wild ramifications is not Lagrangian in general.  In fact, we find an $\ell$-adic sheaf $\cF$ on a smooth fibration $f:X\rightarrow S$ such that, for $s\in S$, the family of singular supports $SS(\cF|_{X_s})$ jumps from non-Lagrangian closed subset of $\bT^*X_s$ to Lagrangian ones at isolated closed points of $S$.

\subsection{}
We expect in the future a generalization of part (iv) of Theorem \ref{introdellau} to higher dimension cases and its relation with vanishing cycles on general bases \cite{Orgogozo}. If we have a positive answer to this generalization,  we could obtain that $f:X\rightarrow S$ is locally acyclic with respect to $\cF$ in  \ref{introsetting} if and only if $f:X\rightarrow S$ is universally locally acyclic with respect to $\cF$. Indeed, we only need to show that locally acyclicity implies universally acyclicity by the help of the generalization. To show the universally locally acyclicity, it is sufficient to show that the map $f_n:\mathbb A^n_X\rightarrow \mathbb A^n_S$ induced by $f:X\rightarrow S$ is locally acyclic with respect to $\cF|_{\mathbb A^n_X}$. If part (iv) of Theorem \ref{introdellau} is valid for higher dimensional situations, there exists a finite and surjective map $\pi:S'\rightarrow S$ and a cycle $B'$ of $\bT^*((X\times_SS')/S')$ such that, for any algebraic geometric point $\bar s'$ of $ S'$, we have $CC(\cF|_{\bar s'})=B'_{\bar s'}$. We put $X'=X\times_SS'$ and we denote by $B'_n$ the pull-back of $B'$ on $\bT^*(\bA^n_{X'}/\bA^n_{S'})$. Then, for any algebraic geometric point $\bar t'$ of $\bA^n_{S'}$, we have $CC(\cF|_{\bar t'})=B'_{\bar t'}=(B'_n)_{\bar t'}$. Using the generalization of part (iv) of Theorem \ref{introdellau} again for  $\cF|_{\bA^n_X}$ and $f_n:\mathbb A^n_X\rightarrow \mathbb A^n_S$, we obtain that $f_n:\mathbb A^n_X\rightarrow \mathbb A^n_S$ is locally acyclic with respect to $\cF|_{\bA^n_X}$.

\subsection{}
There have been several works on the generalization of Deligne and Laumon's semi-continuity of Swan conductors. Saito proved a semi-continuity property for the total dimension of the stalk of vanishing cycles complexes \cite{Saito15-2}. It is an essential step to obtain characteristic cycles of $\ell$-adic sheaves. For an $\ell$-adic sheaf on a smooth variety, one can associate a divisor called the {\it total dimension divisor} using Abbes and Saito's {\it non-logarithmic} ramification filtration of Galois group of local fields \cite{Abbes-Saito02, Saito13}. In \cite{HuYang}, the authors proved the lower semi-continuity for total dimension divisors of $\ell$-adic sheaves on a smooth fibration. 

In the theory of $\mathcal D$-modules, a local invariant of an algebraic $\mathcal D$-module on a smooth complex curve called {\it irregularity} is an analogue of Swan conductors. Deligne proved the semi-continuity for irregularities of relative differential systems on relative curves \cite{letterkatz}. After that,  Andr\'e proved the same property for general differential systems \cite{Andre}. We expect a generalization of their works in higher relative dimension situations for characteristic cycles of $\mathcal D$-modules.
\subsection{}
The article is organized as follows. After notation and geometric preliminaries, we give the notion of relative singular support of \'etale sheaves in \S 4. We prove the generic existence of relative singular support and study its fibers in \S 5. Based on results in \S 5, we prove the generic constancy of characteristic cycles in \S 6. Under certain geometric conditions, we also write down a proposition  in \S 6 that compares the characteristic cycle and the relative characteristic cycle,  which is given by Saito.   In \S 7, we provide two examples which shows that part (i) and part (iii) cannot be generalized to higher dimensional situations.

\subsection*{Acknowledgement}
The authors would like to express their gratitude to T. Saito for sharing ideas and for inspiring discussion. The authors would also like to thank A. Abbes and O. Gabber for their stimulating suggestions. The first author is supported by JSPS postdoctoral fellowship and Kakenhi Grant-in-Aid 15F15727 during his stay at the University of Tokyo and the second author is partially supported by Alexander von Humboldt Foundation for his research at Universit\"at Regensburg and Freie Universit\"at Berlin. Some part of the research was accomplished during their visit at IHES. The authors are grateful to these institutions. 

\section{Notation and conventions}

\subsection{}
Let $X$ be a scheme and $\mathcal E$ a locally free sheaf of $\mathcal O_X$-modules of finite rank on $X$. We call the vector bundle of $\mathcal E$ and denote by $\mathbb V(\mathcal E)$ the spectrum of the quasi-coherent $\mathcal O_X$-algebra $\mathrm{Sym}_{\mathcal{O}_X}(\mathcal E)$. Let $E$ be such a vector bundle on $X$. We denote by $\rP(E)$ the projectivization of $E$. Let $C$ be a constructible subset of $E$. We say that $C$ is {\it conical}, if $C$ is invariant under the canonical $\mathbb G_m$-action on $E$ that fixes the zero-section. For a closed conical subset $C$ of $E$, we denote by $\rP(C)$ the projectivization of $C$, which is a closed subset of $\rP(E)$. We denote by $B(C)$ the scheme theoretic image of $C$ in $X$ by the canonical projection $\pi:E\rightarrow X$ and call it the {\it base} of $C$. Observe that the canonical map $C\rightarrow B(C)$ is surjective. We say a conical subset $C$ is {\it strict} if each irreducible component of $C$ is not contained in the zero-section of $E$. We have a bijection
\begin{equation*}
\{\textrm{strict closed conical subsets of } E\}\rightarrow \{\textrm{closed subsets of }\rP(E)\}, \,C\mapsto \rP(C).
\end{equation*}
We always take reduced induced subscheme structures on a closed conical subset of a bundle and its projectivization.

\subsection{}\label{notationfiber}
Let $f\colon X\rightarrow S$ be a morphism of schemes, $s$ a point of $S$ and $\bar s\rightarrow S$ a geometric point above $s$. We denote by $X_s$ (resp. $X_{\bar s}$) the fiber $X\times_Ss$  (resp. $X\times_S\bar s$). Let $\{Z_i\}_{i\in I}$ be a finite set of closed subschemes of $X$ and $Z=\sum_i m_i[Z_i]$ a cycle supported on $Z$. We denote by $Z_s$ (resp. $Z_{\bar s}$) the cycle $\sum_i m_i[(Z_i)_s]$ (resp. $\sum_i m_i[(Z_i)_{\bar s}]$) on $X_s$ (resp. $X_{\bar s}$).

\subsection{}\label{relcotbundle}
Let $f\colon X\rightarrow S$ be a smooth morphism of Noetherian schemes. We denote by $\mathbb T^\ast(X/S)$ the vector bundle $\mathbb V((\Omega^{1}_{X/S})^{\vee})$ on $X$ and call it {\it the relative cotangent bundle on} $X$ {\it with respect to} $S$. We denote by $\mathbb T^*_X(X/S)$ the zero-section of
$\mathbb T^\ast(X/S)$.
Let $u$ be a point of $X$ and $\bar u\rightarrow X$ a geometric point of $X$ above $u$. We put
\begin{equation*}
\mathbb T^{\ast}_u(X/S)=\mathbb T^{\ast}(X/S)\times_Xu \ \ \ {\rm and}\ \ \ \mathbb T^{\ast}_{\bar u}(X/S)=\mathbb T^{\ast}(X/S)\times_X\bar u.
\end{equation*}

Let $g:Y\rightarrow S$ be a smooth morphism and $f:X\rightarrow S$ an $S$-morphism. Then, we have a canonical morphism
 \begin{equation}\label{diffcot}
 dh:\bT^*(X/S)\times_XY\rightarrow \bT^*(Y/S)
 \end{equation}
induced by the canonical map $h^*\Omega^1_{X/S}\rightarrow \Omega^1_{Y/S}$.

\subsection{}\label{sub:t}
Let $f\colon X\rightarrow S$ be a smooth morphism of Noetherian schemes and $C$ a closed conical subset of $\mathbb T^\ast(X/S)$.
Let $Y$ be an $S$-scheme smooth over $S$ and $h\colon Y\rightarrow X$ an $S$-morphism.  We say that  $h\colon Y\rightarrow X$ is {\it $C$-transversal relative to} $S$ {\it at a geometric point} $\bar y\rightarrow Y$ if for every non-zero vector $\mu\in C_{h(\bar y)}=C\times_X\bar y$, the image $dh_{\bar y}(\mu)\in \mathbb T_{\bar y}^\ast(Y/S)$ is not zero, where $dh_{\bar y}\colon \mathbb T_{\bar y}^\ast(X/S)\rightarrow \mathbb T_{h(\bar y)}^\ast(Y/S)$ is the canonical map. We say that $h\colon Y\rightarrow X$ is $C$-{\it transversal relative to} $S$ if it is $C$-transversal relative to $S$ at every geometric point of $Y$. If $h:Y\rightarrow X$ is $C$-transversal relative to $S$, the topological image of $C\times_XY$ in $\mathbb T^\ast(Y/S)$ by the canonical map $dh:\mathbb T^\ast(X/S)\times_X Y\rightarrow \mathbb T^\ast(Y/S)$ is closed, by the same argument of \cite[Lemma 1.1]{Beilinson15}. We put $h^{\circ}C=dh(C\times_XY)$.

Let $Z$ be an $S$-scheme smooth over $S$ and $g\colon X\rightarrow Z$ an $S$-morphism. We say that $g\colon X\rightarrow Z$ is {\it $C$-transversal relative to} $S$ {\it at a geometric point} $\bar x\rightarrow X$ if for every non-zero vector $\nu\in T_{g(\bar x)}^\ast(Z/S)$, we have $dg_{\bar x}(\nu)\notin C_{\bar x}$, where $dg_{\bar x}\colon\mathbb T^*_{g(\bar x)}(Z/S)\rightarrow \mathbb T^*_{\bar x}(X/S)$ is the canonical map. We say that $g\colon X\rightarrow Z$ is $C$-{\it transversal relative to} $S$ if it is $C$-transversal relative to $S$ at all geometric point of $X$.
If the base $B(C)$ of $C$ is proper over $Z$, we put $g_\circ C:=\rpr_1(dg^{-1}(C))$, where $\rpr_1\colon \mathbb T^\ast (Z/S)\times_Z X\rightarrow \mathbb T^\ast (Z/S)$ denotes the first projection and $dg: \mathbb T^*(Z/S)\times_ZX\rightarrow \mathbb T^*(X/S)$ is the canonical map. It is a closed conical subset of $\mathbb T^*(X/S)$.

When $S$ is spectrum of a field, the definitions above are identical to those in \cite{Beilinson15}. In this case, we follow the terminologies of {\it loc. cit.} and omit "relative to $S$".
 We observe that that $g:X\rightarrow Z$ (resp. $h:Y\rightarrow X$) is $C$-transversal relative to $S$ if and only if for every point $s\in S$, the fiber $g_s:X_s\rightarrow Z_s$ (resp. $h_s:Y_s\rightarrow X_s$) is $C_s$-transversal.
Similarly to \cite[Lemma 1.1]{Beilinson15}, if the map $h:Y\rightarrow X$ (resp. $g:X\rightarrow Z$) is $C$-transversal relative to $S$ at a geometric point $\bar u\rightarrow Y$ (resp. $\bar x\rightarrow X$), then, there exists a Zariski open neighborhood $Y_0$ of $\bar u$ (resp. $X_0$ of $\bar x$) such that $h|_{Y_0}:Y_0\rightarrow X$ (resp. $g|_{X_0}:X_0\rightarrow Z$) is $C$-transversal relative to $S$.

\subsection{}
Let $X$ be an irreducible excellent Noetherian scheme. We say that $X$ is {\it generically dimensional} if, for any open dense subscheme $U$ of $X$, we have $\dim X=\dim U$. For example, an irreducible algebraic variety over a field is generically dimensional. But, the spectrum of an excellent integral local ring of dimension $\geq 1$ is not. For any irreducible excellent Noetherian scheme $Y$,  there exists an open dense subscheme $V$ of $Y$ such that $V$ is generically dimensional. The condition of generic dimension is preserved by taking open dense subsechemes and irreducible closed subschemes.

\subsection{}
We fix a prime number $\ell$ and a finite field $\Lambda$ of characteristic $\ell$. Let $X$ be a Noetherian scheme over $\mathbb Z[1/\ell]$. We denote by $D(X,\Lambda)$ the derived category of complexes of \'etale sheaves of $\Lambda$-modules on $X$ and by $D^b_c(X,\Lambda)$ its full subcategory consisting of objects bounded with constructible cohomologies.


\subsection{}
Let $f:X\rightarrow S$ be a morphism of Noetherian schemes over $\mathbb Z[1/\ell]$ and $\cF$ an object of $D^b_c(X,\Lambda)$. We denote by $E_f(\cF)$ the smallest closed subset of $X$ such that $f: X-E_f(\cF)\rightarrow S$ is universally locally acyclic with respect to $\cF$. Let $S'$ be a Noetherian $S$-scheme. We set $X'=X\times_SS'$, set $f':X'\rightarrow S'$ the base change of $f:X\rightarrow S$ and set $\cF'=\cF|_{X'}$. We have $E_{f'}(\cF')\subseteq E_f(\cF)\times_SS'$.

\subsection{}\label{ncl}
Let $X$ be an irreducible Noetherian scheme over $\mathbb Z[1/\ell]$, $\xi$ the generic point of $X$ and $\cF$ an object on $D^b_c(X,\Lambda)$. We denote by
$C_0(\cF,X)$ the constructible subset
\begin{equation*}
C_0(\cF,X)=\{s\in X\,;\,\mathrm{cosp}:\cF_{\bar s}\xrightarrow{\sim} \cF_{\bar \xi}\}
\end{equation*}
of $X$, where $\bar s$ (resp. $\bar \xi$) denotes a geometric point above $s$ (resp. $\xi$) and $\mathrm{cosp}$ is the canonical cospecialization map. We denote by $NC_0(\cF,X)$ the complement of $C_0(\cF,X)$ in $X$ and by $NC(\cF,X)$ the closure of $NC_0(\cF,X)$ in $X$ and we call it the {\it non-constant locus of} $\cF$ on $X$. We always take the reduced induced subscheme structure on the non-constant locus.

\subsection{}\label{defdimtot}
Let $K$ be a complete discrete valuation field, $\mathcal O_K$ its integer ring, $F$ the residue field of $\mathcal O_K$, $\overline K$  a separable closure of $K$ and $G_K$ the Galois group of $\overline K$ over $K$. We assume that $F$ is perfect of characteristic $p>0$. Let $M$ be a finite generated $\Lambda$-module with continuous $G_K$-action. The {\it total dimension} of $M$ is defined by 
\begin{equation*}
\mathrm{dimtot}_KM=\mathrm{Sw}_KM+\dim_{\Lambda}M,
\end{equation*}
where $\mathrm{Sw}_KM$ denotes the Swan conductor of $M$ \cite[19.3]{Serre97}.

\section{Radon Transforms and Legendre transforms}

\subsection{}\label{projspace} 
We fix an integer $n\geq 1$. We denote by $E_{\mathbb Z}=\bA^{n+1}_{\mathbb Z}$ the trivial vector bundle over $\mathrm{Spec}({\mathbb Z})$ of rank $n+1$, by $E_{\mathbb Z}^{\vee}$ the dual bundle of $E_{\mathbb Z}$ and by $\bP_{\mathbb Z}=\rP(E_{\mathbb Z})$ (resp. $\bP_{\mathbb Z}^{\vee}=\rP(E_{\mathbb Z}^{\vee})$) the projectivization of $E_{\mathbb Z}$ (resp. $E_{\mathbb Z}^{\vee}$). The {\it universal hyperplane} $\bH_{\mathbb Z}$ of $\bP_{\mathbb Z}\times_{\mathbb Z}\bP^{\vee}_{\mathbb Z}$ is the hypersurface defined by the principal Cartier divisor associated to the section
\begin{equation*}
\mathrm{id}\in \mathrm{End}_{\mathbb Z}(\Gamma(\bP_{\mathbb Z}, \cO_{\bP_\bZ}(1)))=\Gamma(\bP_{\mathbb Z}\times_{\mathbb Z}\bP^{\vee}_{\mathbb Z}, \rpr_1^*(\cO_{\bP_\bZ}(1))\otimes\rpr_2^*(\cO_{\bP^{\vee}_\bZ}(1))).
\end{equation*}
For any point $s\in \mathrm{Spec}({\mathbb Z})$, we have 
$(\bH_{\mathbb Z})_{s}=\{(x,x')\in\bP_{s}\times_{s}\bP^{\vee}_{s}\,|\, x\in H_{x'}\}$, where $H_{x'}$ denotes the hyperplane of $\bP_s$ associated to $x'\in\bP^{\vee}_s$. Let $d$ be an integer $\geq 1$ and $N=\binom{n+d}{d}$. We denote by $\wt E_{\mathbb Z}=\bA^N_{\mathbb Z}$ the trivial vector bundle over ${\mathbb Z}$ of rank $N$, by $\wt E_{\mathbb Z}^{\vee}$ the dual bundle of $\wt E_{\mathbb Z}$ and by $\wt\bP_{\mathbb Z}=\rP(\wt E_{\mathbb Z})$ (resp. $\wt\bP^{\vee}_{\mathbb Z}=\rP( \wt E^{\vee}_{\mathbb Z})$) the projectivization of $\wt E_{\mathbb Z}$ (resp. $\wt E_{\mathbb Z}^{\vee})$). Let $\wt i_{\mathbb Z}:\bP_{\mathbb Z}\rightarrow \wt\bP_{\mathbb Z}$ be a Veronese embedding of degree $d$. We denote by $\wt\bH_{\mathbb Z}$ the universal hyperplane of $\wt\bP_{\mathbb Z}\times_{\mathbb Z}\wt\bP_{\mathbb Z}^{\vee}$. 

Let $S$ be a scheme. For any $A_{\mathbb Z}$  ($A_{\mathbb Z}=E_{\mathbb Z}, E^{\vee}_{\mathbb Z}, \wt E_{\mathbb Z},\cdots$) above, we simply put $A=A_{\mathbb Z}\times_{\mathbb Z}S$. We denote by $\wt i:\bP\rightarrow \wt \bP$ the base change of $\wt i_{\mathbb Z}:\bP_{\mathbb Z}\rightarrow \wt\bP_{\mathbb Z}$ and by $\rpr:\bH\rightarrow \bP$ (resp. $\rpr^\vee:\bH\rightarrow\bP^{\vee}$, $\wt\rpr:\wt\bH\rightarrow\wt\bP$ and $\wt\rpr^{\vee}:\wt\bH\rightarrow\wt\bP^{\vee}$) the canonical projection.

\subsection{}\label{legtran}
We have a canonical injection of locally free sheaves of $\cO_{\bP_{\bZ}}$-modules 
\begin{equation*}
\Omega^1_{\bP_{\bZ}/\bZ}(1)\rightarrow \Gamma(\bP_\bZ,\cO_{\bP_\bZ}(1))\otimes_\bZ\cO_{\bP_\bZ}
\end{equation*}
with the cokernel $\cO_{\bP_\bZ}(1)$. It induced an injection of vector bundles
\begin{equation*}
\bV((\Omega^{1}_{\bP_{\bZ}/\bZ})^{\vee}(-1))\rightarrow E^{\vee}_{\bZ}\times_\bZ\bP_{\bZ},
\end{equation*}
hence an injection of their projectivizations 
\begin{equation*}
\rP(\bT^*(\bP_\bZ/\bZ))\rightarrow \bP^{\vee}_\bZ\times_\bZ\bP_\bZ,
\end{equation*}
which gives rise to an isomorphism
\begin{equation*}
\theta_{\bZ}: \rP(\bT^*(\bP_\bZ/\bZ))\xrightarrow{\sim}\bH_\bZ.
\end{equation*}
Exchange $\bP_\bZ$ and $\bP_\bZ^{\vee}$, we obtain an isomorphism
\begin{equation*}
\theta^{\vee}_{\bZ}: \rP(\bT^*(\bP^{\vee}_\bZ/\bZ))\xrightarrow{\sim}\bH_\bZ.
\end{equation*}
The maps $\theta_{\bZ}$ and $\theta_{\bZ}^{\vee}$ give rise to isomorphisms
\begin{equation*}
\theta: \rP(\bT^*(\bP/S))\xrightarrow{\sim}\bH\ \ \ \textrm{and}\ \ \  \theta^{\vee}:\mathrm P(\mathbb T^*(\mathbb P^{\vee}/S))\xrightarrow{\sim}\mathbb H.
\end{equation*}
The isomorphisms $\theta_\bZ$, $\theta^{\vee}_\bZ$, $\theta$ and $\theta^{\vee}$ are called {\it Legendre transforms}.

Let $S$ be a Noetherian scheme and $s$ a point of $S$. We have
\begin{equation*}
\theta_s:\rP(\bT^*\bP_s)\xrightarrow{\sim}\bH_s,\ \ \  (x,\bar\lambda_{x,x'})\mapsto(x,x'),
\end{equation*}
where $(x,x')$ denotes a closed point of $\bH_s$ and $\lambda_{x,x'}$ denotes the $1$-dimensional vector space $\bT^*_x\bP_s$ which is orthogonal to the hyperplane $H_{x'}\subseteq \bP_s$ associated to $x'\in \bP^{\vee}_s$.  For any $(x,x')\in\bH_s$, we have canonical injections
\begin{equation*}
d\rpr_x:\bT^*_x\bP_s\rightarrow \bT^*_{(x,x')}\bH_s, \ \ \ \textrm{and}\ \ \ d\rpr^{\vee}_{x'}:\bT^*_{x'}\bP^{\vee}_s\rightarrow \bT^*_{(x,x')}\bH_s.
\end{equation*}
The intersection of $\bT^*_x\bP_s$ and $\bT^*_{x'}\bP^{\vee}_s$ in $\bT^*_{(x,x')}\bH_s$ is a $1$-dimensional vector space $L_{x,x'}$ and we have
\begin{equation}\label{dpreqn}
d\rpr_x(\lambda_{x,x'})=L_{x,x'}=d\rpr^{\vee}_{x'}(\lambda_{x',x}).
\end{equation}
Let $C$ be a closed conical subset of $\bT^*(\bP/S)$. We denote by $C^{+}$ the union of $C$ and the zero-section of $\mathbb T^*(\mathbb P/S)$. The projectivization $\rP(C)$ of $C$ can be considered as a closed subset of $\mathrm P(\mathbb T^*(\mathbb P^{\vee}/S))$ by Legnedre transforms. We denote by $C^{\vee}$ the strict closed conical subset of $\mathbb T^*(\mathbb P^{\vee}/S)$ associated to $\rP(C)$. We take the same notation after exchanging $\bP$ and $\bP^{\vee}$. Obviously, we have $C^+=C^{\vee\vee+}$ and, if $C$ is strict, we have $C=C^{\vee\vee}$. By  \eqref{dpreqn}, we have $\rpr^{\vee}_{\circ}\rpr^{\circ}(C^+)=C^{\vee+}$.

\begin{lemma}\label{gendiml}
 Let $X$ and $Y$ be irreducible excellent Noetherian schemes and
 $f:X\rightarrow Y$ a dominant morphism of finite type. Then, there exists a Zariski open dense subscheme $V$ of $Y$ such that $X_V=X\times_YV$ is generically dimensional.
 \end{lemma}
\begin{proof}
After replacing $Y$ by a Zariski open dense subscheme, we may assume that $Y$ is generically dimensional. By \cite[II, 6.9.1]{EGA4}, we may further assume that $f_{\mathrm{red}}:X_{\mathrm{red}}\rightarrow Y_{\mathrm{red}}$ is flat. Let $\eta$ be the generic point of $Y$. By \cite[III, 9.5.6]{EGA4}, we may assume that, for any $y\in Y$, each irreducible component $X_v$ has dimension $\dim_{k(\eta)} X_{\eta}$. We have $\dim X=\dim Y+\dim_{k(\eta)}X_{\eta}$.
For any open dense subscheme $U$ of $X$, the map $f:U_{\mathrm{red}}\rightarrow Y_{\mathrm{red}}$ is flat. Moreover, for any $y\in Y$, the fiber $U_y$ is empty or each irreducible component of $U_y$ has dimension $\dim_{k(\eta)} X_{\eta}$. Hence, $\dim U=\dim Y+\dim_{k(\eta)}X_{\eta}=\dim X$.
\end{proof}



\subsection{}(\cite[4.2]{Beilinson15}).\label{wlint}
Let $k$ be a field and we assume that all $k$-schemes in this section are of finite type over $\mathrm{Spec}(k)$.
 Let $P$ be an irreducible $k$-scheme and $\pi: H\rightarrow P$ a morphism of $k$-schemes. We denote by $H_P^{(2)}$ the complement of the diagonal in $H\times_PH$. We say that two irreducible closed subschemes $Z_1$ and $ Z_2$ of $H$ {\it intersect well relatively to} $\pi: H\rightarrow P$ if $(Z_1\times_P Z_2)\bigcap H_P^{(2)}$ is equidimensional and
  \begin{equation*}
  \dim_k\left((Z_1\times_P Z_2)\cap H_P^{(2)} \right)= \dim_k Z_1+\dim_k Z_2-\dim_k P.
  \end{equation*}
 Let $Z$ be an irreducible $k$-scheme and $g:Z\rightarrow P$ a generically surjective morphism. The map $g:Z\rightarrow P$ is called {\it small} if $\dim_k((Z\times_P Z)- \delta(Z))<\dim_k Z$, where $\delta:Z\rightarrow Z\times_PZ$ denotes the diagonal map. A small map is generically radicial.


We assume that $\pi:H\rightarrow P$ is proper.  If irreducible closed subschemes $Z_1$ and $Z_2$ of $H$ intersect well relatively to $\pi:H\rightarrow P$, if $Z_1$ is generically finite over $\pi(Z_1)$ and if $\dim_k Z_2<\dim_k P$, then $Z_1\subseteq Z_2$ is equivalent to $\pi(Z_1)\subseteq \pi(Z_2)$ (\cite[Lemma 4.2]{Beilinson15}).
If an irreducible closed subscheme $Z$ of $H$ intersects well relatively to $\pi:H\rightarrow P$ with itself and $\dim_k Z<\dim_k P$, then the map $\pi|_Z:Z\rightarrow \pi(Z)$ is small and generically radicial ({\it loc. cit.}).

\subsection{}\label{genericsmall}
Let $T$ be an irreducible excellent Noetherian scheme and $P$ an irreducible $T$-scheme dominant and of finite type over $T$.  We assume that $P$ is generically dimensional. Let $\pi: H\rightarrow P$ be a morphism of finite type.  We denote by $H_P^{(2)}$ the complement of the diagonal in $H\times_PH$. Let $Z_1$ and $Z_2$ be irreducible closed subschemes of $H$ which are dominant over $T$. After replacing $T$ by a Zariski open dense subscheme, we may assume that all schemes above are generically dimensional (Lemma \ref{gendiml}). We say that $Z_1$ and $ Z_2$ {\it generically intersect well relatively to} $\pi: H\rightarrow P$ if, after replacing $T$  by a Zariski open dense subscheme, each irreducible component of $(Z_1\times_P Z_2)\bigcap H_P^{(2)}$ is generically dimensional, has same dimension, and
  \begin{equation*}
  \dim\left((Z_1\times_P Z_2)\cap H_P^{(2)} \right)= \dim Z_1+\dim Z_2-\dim P.
  \end{equation*}
Let $Z$ be an irreducible excellent Noetherian scheme and $g:Z\rightarrow P$ a generically surjective morphism. We may assume that $Z$ is generically dimensional after replacing $T$ by a Zariski open dense subset. The map $g:Z\rightarrow P$ is called {\it generically small} if, after replacing $T$ by a Zariski open dense subscheme, we have $\dim((Z\times_P Z)- \delta(Z))<\dim Z$, where $\delta:Z\rightarrow Z\times_PZ$ denotes the diagonal map. A generically small map is generically radicial.

We assume that $\pi:H\rightarrow P$ is proper. Let $Z_1$ and $Z_2$ be irreducible closed subschemes of $H$ which are dominant over $T$. We may assume that $Z_1$ and $Z_2$ are generically dimensional after replacing $T$ by a Zariski open dense subset. If $Z_1$ and $Z_2$ generically intersect well relatively to $\pi:H\rightarrow P$, if $Z_1$ is generically finite over $\pi(Z_1)$ and if $\dim Z_2<\dim P$, then $Z_1\subseteq Z_2$ is equivalent to $\pi(Z_1)\subseteq \pi(Z_2)$. Let $Z$ be an irreducible closed subscheme of $H$ which is generically dimensional and flat over $T$.
If $Z$ generically intersects well relatively to $\pi:H\rightarrow P$ with itself and $\dim Z<\dim P$, then the map $\pi|_Z:Z\rightarrow \pi(Z)$ is generically small and generically radicial. 
Let $\eta$ be the generic point of $T$. The two assertions above are deduced by using \cite[Lemma 4.2]{Beilinson15} to the generic fiber $\pi_{\eta}:H_\eta\rightarrow P_\eta$,

\begin{lemma}\label{fibradicial}
  Let $T$ be an irreducible Noetherian schemes, $X$ and $Y$ irreducible $T$-schemes flat and of finite type over $T$ and $g:X\rightarrow Y$ a surjective and generically radicial $T$-morphism. For any $s\in T$, we denote by $\{(X_s)_\alpha\}_{\alpha\in I(s)}$ (resp. $\{(Y_s)_\beta\}_{\beta\in J(s)}$) the irreducible components of $X_s$ (resp. $Y_s$). Then, there exists an open dense subset $W$ of $T$ such that, for any $s\in W$, the map $g_s:X_s\rightarrow Y_s$ induces an identity between $I(s)$ and $J(s)$ and, for any $\alpha\in I(s)$, the map $g_s:(X_s)_\alpha\rightarrow(Y_s)_\alpha$ is surjective and generically radicial.
\end{lemma}
\begin{proof}

Since $g:X\rightarrow Y$ is of finite type and generically radicial, we have a Cartesian diagram
  \begin{equation*}
    \xymatrix{\relax
    U\ar[d]_{j'}\ar[r]^{g'}\ar@{}|-{\Box}[rd]&V\ar[d]^j\\
    X\ar[r]_g&Y}
  \end{equation*}
  where $V$ is an open dense subset of $Y$, $j:V\rightarrow Y$ is the canonical injection and $g':U\rightarrow V$ is surjective and radicial. Hence, $U$ is an open dense subset of $X$.  By \cite[III, 9.6.1]{EGA4}, there exists an open dense subset $W\subseteq T$ such that, for any $s\in W$, we have $\overline U_s=X_s$ and $\overline V_s=Y_s$. Hence $U_s$ (resp. $V_s$) contains all generic points of irreducible components of $X_s$ (resp. $Y_s$). Since, for any $s\in W$, $g'_s:U_s\rightarrow V_s$ is surjective and radicial, we have a one-to-one correspondence of the generic points of irreducible components of $X_s$ and of $Y_s$, i.e., $I(s)=J(s)$. For any $s\in W$ and any $\alpha\in I(s)$, the map $g_s:(X_s)_\alpha\rightarrow (Y_s)_\alpha$ is generically radicial since $g'_s:((X_s)_\alpha\bigcap U_s)\rightarrow ((Y_s)_\alpha\bigcap V_s)$ is radicial.
\end{proof}

\begin{lemma}\label{lem:ab-4}
Let $T$ be an irreducible Noetherian scheme, $X$ and $Y$ integral $T$-schemes dominant and of finite type over $T$, and $g:X\rightarrow Y$ a $T$-morphism. We assume that $g$ is  dominant, generically finite, generically radicial and that the fibers $X_t$ and $Y_t$ are geometrically integral for any $t\in T$. We denote by $\zeta$ (resp. $\eta$) the generic points of $X$ (resp. $Y$) and,
for any geometric point $\bar t$ of $T$, by $\zeta_{\bar t}$ (resp. $\eta_{\bar t}$) the generic points of $X_{\bar t}$ (resp. $Y_{\bar t}$). Then, there exists a Zariski open dense subset $W$ of $T$ such that,  $[\zeta_{\bar t}:\eta_{\bar t}]=[\zeta:\eta]$ for any geometric point $\bar t\rightarrow W$. In particular, if the generic point of $T$ is of characteristic $0$, then the canonical map $g_{\bar t}:X_{\bar t}\rightarrow Y_{\bar t}$ is birational for any $\bar t\rightarrow W$.
\end{lemma}
\begin{proof}
Since $g:X\rightarrow Y$ is dominant, generically finite and generically radicial,  there is an open dense subset  $V$  of $Y$ such that  $g':U=V\times_Y X\rightarrow V$ is surjective, finite, flat and radicial, where $g'$ is the base change of $g$  by $V\hookrightarrow Y$.
Let $W$ be an open dense subset of $T$ such that the fiber $V_t$ is non-empty for any $t\in W$. Then, for any geometric point $\bar t$ of $W$,  we have $\zeta_{\bar t}=\eta_{\bar t}\times_{V_{\bar t}}U_{\bar t}=\eta_{\bar t}\times_VU$, hence $[\zeta_{\bar t}:\eta_{\bar t}]=[U:V]=[\zeta:\eta]$.

When the generic point of $T$ is of characteristic $0$, we have $[\zeta:\eta]=1$, i.e., $g:V\rightarrow U$ is an isomorphism.
\end{proof}

\begin{lemma}[{cf. \cite[Lemma 4.3]{Beilinson15}}]\label{interwell}
We take the notation and assumptions of \ref{projspace}, we assume that $S$ is an irreducible excellent Noetherian scheme and that the Veronese embedding $\wt i:\bP\rightarrow\wt\bP$ in \ref{projspace} has degree $d\geq 3$. Let $C_1$ and $C_2$ be irreducible closed conical subsets of $\ct(\bp/S)$ which are dominant over $S$. We consider $\rP(\wt i_{\circ} C_1)$ and $\rP(\wt i_{\circ} C_2)$ as closed subschemes of $\wt\bH$ by Legendre transform (\ref{legtran}). Then, after replacing $S$ by a Zariski open dense subset,
$\rP(\wt i_{\circ} C_1)$ and $\rP(\wt i_{\circ} C_2)$ generically intersect well relatively to $\wt\rpr^\vee:\wt\bH\rightarrow\wt\bP^{\vee}$.
\end{lemma}
\begin{proof}
We may assume that $S$ is integral. By Lemma \ref{gendiml}, we may assume that all schemes are generically dimensional after shrinking $S$. By \cite[II, 6.9.1]{EGA4}, we may assume that $C_1$ and $C_2$ are flat over $S$ after shrinking $S$. Since the canonical morphism $d\wt i:\bT^*(\wt\bP/S)\times_{\wt\bP}\bP\rightarrow\bT^*(\bP/S)$ is smooth, conical subschemes $\wt i_{\circ}C_1$ and $\wt i_{\circ}C_2$ are flat over $S$. Hence, $\rP(\wt i_{\circ}C_1)$ and $\rP(\wt i_{\circ}C_2)$ are flat over $S$. Let  $\wt\delta:\wt\bH\rightarrow\wt\bH\times_{\wt\bP^{\vee}}\wt\bH$
be the diagonal map and
 $\wt\bH^{(2)}_{\wt\bP^{\vee}}$ the complement of $\wt\delta(\wt\bH)$ in $\wt\bH\times_{\wt\bP^{\vee}}\wt\bH$.
Shrinking $S$ again, we may assume that each irreducible component of
\begin{equation*}
Y=\left(\rP(\wt i_{\circ}C_2)\times_{\wt\bP^\vee}\rP(\wt i_{\circ}C_2)\right)\bigcap\wt\bH^{(2)}_{\wt\bP^{\vee}}
\end{equation*}
is flat over $S$. Let $\eta$ be the generic point of $S$. There is a one-to-one correspondence between the set of irreducible components of $Y$
and that of $Y_{\eta}$. By \cite[Lemma 4.3]{Beilinson15}, $(\rP(\wt i_{\circ}C_1))_\eta$ and $(\rP(\wt i_{\circ}C_2))_\eta$ intersect well relatively to $\wt\rpr_\eta:\wt\bH_\eta\rightarrow\wt\bP^{\vee}_\eta$.
Hence, $Y_\eta$ is equidimensional and we have
\begin{equation}\label{genewlinter}
 \dim_{k(\eta)}Y_\eta=\dim_{k(\eta)}(\rP(\wt i_{\circ}C_1))_\eta+\dim_{k(\eta)}(\rP(\wt i_{\circ}C_2))_\eta-\dim_{k(\eta)}\wt\bP^{\vee}_{\eta}.
\end{equation}
Hence, $Y$ is also equidimensional and \eqref{genewlinter} implies that
\begin{equation*}
  \dim Y=\dim (\rP(\wt i_{\circ}C_1))+\dim (\rP(\wt i_{\circ}C_2))-\dim \wt\bP^{\vee}.
\end{equation*}
Hence, $\rP(\wt i_{\circ} C_1)$ and $\rP(\wt i_{\circ} C_2)$ generically intersect well relatively to $\wt\rpr^\vee:\wt\bH\rightarrow\wt\bP^{\vee}$.
\end{proof}

\begin{proposition}[{cf. \cite[Proposition 4.5]{Beilinson15}}]\label{geomctod}
We take the notation and assumptions of \ref{projspace}, we assume that $S$ is an irreducible generically dimensional excellent Noetherian scheme and that the Veronese embedding $\wt i:\bP\rightarrow\wt\bP$ in \ref{projspace} has degree $d\geq 3$. Let $C$ be a closed conical subset of $\bT^*(\bP/S)$ such that $C_\eta\neq\emptyset$ is dominant over $S$ and that $\dim C\leq\dim \bP$. Then,
\begin{enumerate}
\item If $C$ is irreducible, then the map $\wt\rpr^{\vee}: \rP(\wt i_\circ C)\rightarrow \wt\rpr^{\vee}(\rP(\wt i_\circ C))$ is generically small after replacing $S$ by a Zariski open dense subscheme.
\item After replacing $S$ by a Zariski open dense subscheme, there is a one-to-one correspondence between the sets of irreducible components of $C$ and $\wt\rpr^{\vee}(\rP(\wt i_\circ C))$.
\end{enumerate}\end{proposition}

\begin{proof}
(1) By Lemma \ref{gendiml}, we may assume that $C$ and $\rP(\wt i_\circ C)$ are generically dimensional after shrinking $S$.
We have $\dim \rP(\wt i_\circ C)\leq\dim S+N-2$. Hence $\dim \rP(\wt i_\circ C)<\dim\wt\bP$. After shrinking $S$, the irreducible scheme
$\rP(\wt i_\circ C)\subseteq\wt\bH$ generically intersects well relatively to $\wt\bP^{\vee}$ with itself (Lemma \ref{interwell}). Hence, $\wt\rpr^{\vee}: \rP(\wt i_\circ C)\rightarrow \wt\rpr^{\vee}(\rP(\wt i_\circ C))$ is generically small, hence generically radicial (cf. \ref{genericsmall}).

(2) Let $\{C_v\}_{v\in I}$ be set of the irreducible component of $C$. There is a one-to-one correspondence between $\{C_v\}_{v\in I}$ and $\{\rP(\wt i_{\circ}C_v)\}_{v\in I}$. After shrinking $S$, we may assume that,  for each $v\in I$, we have $(C_v)_\eta\neq \emptyset$ and $\rP(\wt i_{\circ} C_v)$ is generically dimensional and that, for each pair $v,v'\in I$ ($v$ and $v'$ can be the same), $\rP(\wt i_{\circ} C_v)$ and $\rP(\wt i_{\circ} C_{v'})$ intersect will relatively to $\wt\bP^{\vee}$ (Lemma \ref{interwell}).
 Since $\rP(\wt i_{\circ}C_v)$ $(v\in I)$ are distinct irreducible subsets of $\wt\bH$, $\wt\rpr^{\vee}(\rP(\wt i_{\circ}C_v))$ $(v\in I)$ are distinct irreducible subsets of $\wt\bP^{\vee}$ (\ref{genericsmall}). Hence, we have a one-to-one correspondence between $\{\rP(\wt i_{\circ}C_v)\}_{v\in I}$ and $\{\wt\rpr^\vee(\rP(\wt i_{\circ}C_v))\}_{v\in I}$ by the projection $\wt\rpr^\vee:\wt\bH\rightarrow\wt\bP^{\vee}$. We obtain (2).
\end{proof}

\subsection{}\label{radontran}
We take the notation and assumptions of \ref{projspace} and we assume that $\ell$ is invertible in $S$. We define the {\it Radon transform} by the functor \cite{Bry}
\begin{equation*}
R_S:D^b_c(\mathbb P,\Lambda)\rightarrow D^b_c(\mathbb P^{\vee},\Lambda),\ \ \ R_S(\mathcal F)=\mathrm R \mathrm{pr}^{\vee}_*(\mathrm{pr}^*\mathcal F)[n-1],
\end{equation*}
and we define the {\it dual Radon transform} by the functor
\begin{equation*}
R^{\vee}_S:D^b_c(\mathbb P^{\vee},\Lambda)\rightarrow D^b_c(\mathbb P,\Lambda),\ \ \ R_S^{\vee}(\mathcal G)=\mathrm R \mathrm{pr}_*(\mathrm{pr}^{\vee *}\mathcal G)(n-1)[n-1].
\end{equation*}
Since $\rpr:\bH\rightarrow \bP$ and $\rpr^\vee:\bH\rightarrow\bP^\vee$ are projective and smooth, The functor $R_S$ is both left and right adjoint to $R_S^\vee$ by Poincar\'e duality and proper base change theorem.
We denote by $\wt R_S$ and $\wt R^\vee_S$
the Radon transform and dual Radon transform for
the pair of morphisms $(\wt\rpr,\wt\rpr^\vee):\wt \bP\leftarrow \wt\bH\rightarrow\wt\bP^{\vee}$, respectively.

\begin{proposition}[{cf. \cite[1.6.1]{Beilinson15}}]\label{bei1.6.1}
We take the notation and assumptions of \ref{projspace} and we assume that $\ell$ is invertible in $S$. Let $\cF$ be an object of $D^b_c(\bP,\Lambda)$ and $\cG$ an object of $D^b_c(\bP^\vee,\Lambda)$.  Then, the mapping cone of the adjunction map $\cF\rightarrow (R^\vee_S\circ R_S)(\cF)$ (resp. $(R_S\circ R^\vee_S)(\cG)\rightarrow \cG$) has locally constant cohomologies, after replacing $S$ by a Zariski open dense subset.

\end{proposition}
\begin{proof}
The proposition is valid when $S$ is spectrum of a field \cite[1.6.1]{Beilinson15}. Hence, (1) and (2) still hold since we allow shrinking $S$.
\end{proof}

\begin{lemma}\label{nclfiber}
 Let $S$ and $X$ be irreducible Noetherian schemes and $f:X\rightarrow S$ a flat morphism of finite type with irreducible fibers. We assume that $\ell$ is invertible in $X$. Let $\cF$ be an object of $D^b_c(X,\Lambda)$. Then, there exists a Zariski open dense subset $V$ of $S$ such that, for any $s\in V$, we have (\ref{ncl})
\begin{equation*}
(NC(\cF,X)\times_Ss)_{\red}=NC(\cF|_{X_s},X_s).
\end{equation*}
\end{lemma}
\begin{proof}
We simply put $Z_0=NC_0(\cF,X)$ and put $Z=NC(\cF,X)$. The canonical map $i:Z_0\rightarrow Z$ is dominant. By \cite[III, 9.6.1]{EGA4}, there exists a Zariski open dense subset $V$ of $S$ such that for any $s\in V$, the map $i_s:(Z_0)_s\rightarrow Z_s$ is dominant.
For any $s\in V$, we have $(Z_0)_s=NC_0(\cF|_{X_s}, X_s)$ and $\overline{(Z_0)_s}=NC(\cF|_{X_s}, X_s)$. Hence, for any $s\in V$, we have 
\begin{equation*}
(Z_s)_{\mathrm{red}}=\overline{(Z_0)_s}=NC(\cF|_{X_s}, X_s).
\end{equation*}
\end{proof}

\section{Relative singular supports}
In this section, let $S$ be a connected Noetherian scheme,  $f:X\rightarrow S$ a smooth morphism of finite type and
$C$ a closed conical subset in the relative cotangent bundle $\mathbb T^\ast(X/S)$.

\subsection{}
A {\it test pair of} $X$ {\it relative to} $S$ is a pair of morphisms $(g,h): Y\leftarrow U\rightarrow X$ such that $U$ and $Y$ are $S$-schemes smooth and of finite type over $S$ and $g:U\rightarrow Y$ and $h:U\rightarrow X$ are $S$-morphisms. We say that $(g,h)$ is $C$-{\it transversal relative to} $S$ if $h:U\rightarrow X$ is $C$-transversal relative to $S$ and $g:U\rightarrow Y$ is $h^{\circ}C$-transversal relative to $S$.

We say that a test pair $(g,h):Y\leftarrow U\rightarrow X$ relative to $S$ is {\it weak} if it satisfies: (a) $Y=\mathbb A^1_S$ is an affine line over $S$ and (b) the morphism $h:U\rightarrow X$ is
\begin{itemize}
\item{} a composition of $\rpr_1:U=V\times_SS'\rightarrow V$ and $h':V\rightarrow X$, where $S'$ is finite and \'etale over $S$ and $h':V\rightarrow X$ is an open immersion, if $S_{\mathrm{red}}$ is a spectrum of a finite field;
\item{}  an open immersion in other situations.
\end{itemize}


\subsection{}
In the following of this section, we assume that all schemes are over $\bZ[1/\ell]$. Let $\mathcal F$ be an object in $D^b_c(X,\Lambda)$. We say that a test pair $(g,h):Y\leftarrow U\rightarrow X$ relative to $S$ is $\mathcal F$-{\it acyclic} if $g:U\rightarrow Y$ is locally acyclic with respect to $h^*\mathcal F$.

\subsection{}
We say that $\mathcal F$ is {\it micro-supported on} $C$ {\it relative to} $S$ if every $C$-transversal test pair of $X$ relative to $S$ is $\mathcal F$-acyclic.  We define  $\mathcal C(\mathcal F, X/S)$  the set
\begin{equation*}
\{C'\subseteq \mathbb T^*(X/S)\,|\,C'\;\textrm{is closed conical and}\;\mathcal F\;\, \textrm{is micro-supported on}\;C'\; \textrm{relative to}\; S\}.
\end{equation*}
We will see that $\mathcal C(\mathcal F,X/S)$ is non-empty if $f:X\rightarrow S$ is universally locally acyclic relative to $\mathcal F$ (cf. Proposition \ref{lem:nonempty}).
 We denote by $\mathcal C^{\mathrm{min}}(\mathcal F, X/S)$ the set of the minimal elements of $\mathcal C(\mathcal F, X/S)$. 

We say that $\mathcal F$ is {\it weakly micro-supported on} $C$ {\it relative to} $S$ if every $C$-transversal weak test pair $(g,h)$ relative to $S$ which is $C$-transversal relative to $S$ is $\mathcal F$-acyclic. We define  $\mathcal C^w(\mathcal F, X/S)$  the set
\begin{equation*}
\{C'\subseteq \mathbb T^*(X/S)\,|\,C'\;\textrm{is closed conical and}\;\mathcal F\;\, \textrm{is weakly micro-supported on}\;C'\; \textrm{relative to}\; S\}.
\end{equation*}
Note that $\mathcal C(\mathcal F, X/S)\subseteq \mathcal C^w(\mathcal F, X/S)$.
Let $C_1$ and $C_2$ be two elements of $\mathcal C^w(\mathcal F, X/S)$. We choose a weak test pair $(g,h)$ of $X$ relative to $S$ which is $(C_1\cap C_2)$-transversal relative to $S$. Then $g:U\rightarrow Y$ is $h^{\circ}(C_1\cap C_2)$-transversal relative to $S$.  For each geometric point $\bar u$ of $U$,  $g$ is either $h^{\circ}C_1$-transversal relative to $S$ at $\bar u$ or $h^{\circ}C_2$-transversal relative to $S$ at $\bar u$, since $\mathbb T^*(Y/S)$ is a line bundle on $Y$. Then, $U$ can be covered by two Zariski open subsets $U_1$ and $U_2$ such that $g|_{U_i}:U_i\rightarrow Y$ is $C_i$-transversal relative to $S$ $(i=1,2)$. Since $C_1, C_2\in \mathcal C^w(\mathcal F, X/S)$, the map $g_i$ $(i=1,2)$ is locally acyclic with respect to $\mathcal F$, i.e., the map $g$ is locally acyclic with respect to $\mathcal F$. It implies that the weak test pair $(g,h)$ is $\mathcal F$-acyclic. We deduce that $C_1\cap C_2\in \mathcal C^w(\mathcal F, X/S)$. In conclusion, the set $\mathcal C^w(\mathcal F, X/S)$ has a smallest element if it is non-empty. 

\subsection{}\label{sub:rss}
If $\mathcal C(\mathcal F, X/S)$ has a smallest element, we denote it by $SS(\mathcal F, X/S)$ and call it the {\it singular support} of $\mathcal F$ {\it relative to} $S$.  We call the smallest element of $\mathcal C^w(\mathcal F, X/S)$ the {\it weak singular support} of $\mathcal F$ relative to $S$ and denote it by $SS^w(\mathcal F,X/S)$.
If $SS(\cF,X/S)$ exists, then $SS^{\omega}(\mathcal F, X/S)\subseteq SS(\mathcal F, X/S)$.

Notice that the relative singular support and the weak relative singular support are invariant after taking reduced induced subscheme for $S$ and $X$.

\begin{example}\label{ex:zero}
Assume that $f=\mathrm{id}_X$. Then $\mathbb T^\ast(X/X)=X$ and the $X$-test pair $(\mathrm{id},\mathrm{id}):X\leftarrow X\rightarrow X$ is $\mathbb T^\ast(X/X)$-transversal relative to $X$. Hence, for an object $\cF$ of $D^b_c(X,\Lambda)$, the relative singular support $SS(\mathcal F, X/X)$ exists if and only if $\mathcal F$ has locally constant cohomologies.
\end{example}

\begin{remark}
All definitions above were firstly introduced by Beilinson when $S=\mathrm{Spec}(k)$ \cite{Beilinson15}.  In this case, we will omit the phase "relative to $S$"  and we abbreviate the notion $SS(\cF, X/S)$ (resp. $SS^w(\cF,X/S)$) by $SS(\cF)$ (resp. $SS^w(\cF)$) to fit the notation in {\it loc. cit.}. Beilinson proved that $SS(\cF)$ exists, that $SS(\cF)=SS^w(\cF)$ and that, when $X$ is equidimensional, $SS(\cF)$ is of equidimension $\dim_kX$ ({\it loc. cit.}).  
\end{remark}

\begin{lemma}\label{lem:1}
Let $U$ and $Y$ be $S$-schemes smooth and of finite type over $S$ and let $h:U\rightarrow X$ and $g:X\rightarrow Y$ be $S$-morphisms.
\begin{itemize}
\item[(1)] If $h:U\rightarrow X$ is smooth, then, for any closed conical subset $C$ of $\mathbb T^\ast (X/S)$, the morphism $h:U\rightarrow X$ is $C$-transversal relative to $S$. Moreover we have $h^\circ C=C\times_X U$.

\item[(2)]  If $g\colon X\rightarrow Y$ is $C$-transversal relative to $S$, then it is smooth on a Zariski neighborhood of the base $B(C)$.
\end{itemize}
\end{lemma}

\begin{proof}If $h$ is smooth, then $dh\colon \mathbb T^\ast(X/S)\times_X U\rightarrow \mathbb T^\ast(U/S)$ is injective. We obtain (1) by definition.

Let $x$ be a point of $B(C)$ and $\bar x\rightarrow B(C)$ a geometric point above $x$.
If $g\colon X\rightarrow Y$  is $C$-transversal relative to $S$, then the canonical map $dg_{\bar x}\colon\mathbb T^\ast_{g(\bar x)}(Y/S)\rightarrow \mathbb T_{\bar x}^\ast(X/S)$ is injective. It is equivalent to that the canonical morphism of coherent $\mathcal O_X$-modules $g^*\Omega^1_{Y/S}\rightarrow \Omega^1_{X/S}$ is injective at $x$. By \cite[2.5.7]{Fu15}, we deduce that $g$ is smooth on a Zariski neighborhood of $x$ in $X$. Hence, we obtain (2).
\end{proof}

\begin{lemma}\label{lem:2}
Let $(g,h): Y\leftarrow U\rightarrow X$ be a test pair of $X$ relative to $S$.
\begin{itemize}
\item[(1)] Assume that $C=\mathbb T^\ast_X(X/S)$. Then $(g, h)$ is $C$-transversal relative to $S$ if and only if $g$ is smooth.
\item[(2)] Assume that $C=\mathbb T^\ast(X/S)$.  Then $(g,h)$ is $C$-transversal relative to $S$ if and only if the canonical map $h\times g\colon U\rightarrow X\times_S Y$ is smooth.
\end{itemize}
\end{lemma}
\begin{proof}
(1) When $h:U\rightarrow X$ is $C$-transversal relative to $S$, the pull-back $h^\circ C$ is the zero-section of $\mathbb T^*(U/S)$. By definition, $g:U\rightarrow Y$ is $h^{\circ}C$-transversal relative to $S$ if and only if, for any geometric point $\bar u\rightarrow U$, the canonical map $dg_{\bar u}:\mathbb T^*_{g(\bar u)}(Y/S)\rightarrow \mathbb T^*_{\bar u}(U/S)$ is injective.
It is equivalent to that the canonical morphism of $\mathcal O_X$-modules $\mathcal O_U$-modules $g^*\Omega^1_{Y/S}\rightarrow \Omega^1_{U/S}$ is injective. It equals to that $g:U\rightarrow Y$ is smooth \cite[2.5.7]{Fu15}.

(2) If $h\times g\colon U\rightarrow X\times_S Y$ is smooth, then $h$ and $g$ are also smooth. By Lemma \ref{lem:1}(1), the morphism $h$ is $C$-transversal relative to $S$ and $h^\circ C=\mathbb T^\ast(X/S)\times_X U$. Since $h\times g$ is smooth, the canonical morphism
\begin{equation}\label{eq:lem:2:1}
d(h\times g):\ct((X\times_SY)/S)\times_{(X\times_SY)}U\rightarrow \ct(U/S)
\end{equation}
is injective. Moreover, the canonical projections $\rpr_1:X\times_SY\rightarrow X$ and $\rpr_2:Y\times_SX\rightarrow Y$ induces two injections
\begin{align*}
d\mathrm{pr}_1:\mathbb T^*(X/S)\times_X(X\times_SY)\rightarrow \mathbb T^*((X\times_SY)/S),\\
d\mathrm{pr}_2:\mathbb T^*(Y/S)\times_Y(X\times_SY)\rightarrow \mathbb T^*((X\times_SY)/S).
\end{align*}
The intersection of the images of $d\mathrm{pr}_1$ and $d\mathrm{pr}_2$ is contained in the zero-section of $T^*(X\times_SY/S)$.
Combining \eqref{eq:lem:2:1}, it implies that the intersection of  the images of $dh:\mathbb T^*(X/S)\times_XU\rightarrow  \mathbb T^*(U/S)$ and $dg:\mathbb T^*(Y/S)\times_XU\rightarrow  \mathbb T^*(U/S)$ is contained in the zero section of $\mathbb T^*(T/S)$.
Hence $g$ is $h^\circ C$-transversal relative to $S$, i.e., the test pair $(g,h)$ is $C$-transversal relative to $S$.

Conversely, we assume that $(g,h)$ is $C$-transversal relative to $S$. It implies that $dh:\mathbb T^*(X/S)\times_XU\rightarrow  \mathbb T^*(U/S)$ is injective and that $g:U\rightarrow Y$ is $h^\circ C$-transversal relative to $S$.  The map $dh$ is injective implies that $h$ is smooth.
By Lemma \ref{lem:1}(2), $g$ is also smooth since the base $B(h^\circ C)$ is equal to $U$. Notice that $h^\circ C=\mathrm{im}(dh)$. Hence, $g:U\rightarrow Y$ is $h^\circ C$-transversal relative to $S$ implies that the intersection of $\mathrm{im}(dh)$ and $\mathrm{im}(dg)$ in $\bT^*(U/S)$ is $\bT^*_U(U/S)$. It implies that the canonical map
\begin{equation*}
d(h\times g):\mathbb T^*(X\times_SY/X)\times_{(X\times_SY)}U\rightarrow \mathbb T^*(U/S)
\end{equation*}
is injective, Hence, $h\times g:U\rightarrow X\times_SY $ is smooth.
\end{proof}

\begin{proposition}\label{lem:nonempty}
If $f\colon X\rightarrow S$ is universally locally acyclic with respect to $\mathcal F$, then $\mathbb T^\ast (X/S)$ is an element of $\mathcal C(\mathcal F, X/S)$.
\end{proposition}

\begin{proof}
Let $(g,h): Y\leftarrow U\rightarrow X$ be a test pair of $X$ relative to $S$ which is $\ct(X/S)$-transversal relative to $S$. We need to show that it is $\mathcal F$-acyclic. Consider the following diagram with Cartesian squares
\begin{equation*}
\xymatrix{
U\ar[dr]_g\ar[r]^-(0.5){h\times g}&X\times_S Y\ar@{}[rd]|-(0.5){\Box}\ar[d]_{\mathrm{pr}_2}\ar[r]^-(0.5){\mathrm{pr}_1}&X\ar[d]^f\\
&Y\ar[r]&S
}
\end{equation*}
Since $f:X\rightarrow S$ is universally locally acyclic with respect to $\mathcal F$, the second projection $\mathrm{pr}_2$ is also universally locally acyclic with respect to $\mathrm{pr}_1^\ast\mathcal F$.  By Lemma \ref{lem:2}(2), the pair $(g,h)$ is $\mathbb T^*(X/S)$-transversal implies that the morphism $h\times g$ is smooth. By the smooth base change theorem, the composition $g=\mathrm{pr}_2 \circ  (h\times g) $ is universally locally acyclic relatively to $h^\ast\mathcal F$, i.e., the pair $(g,h)$ is $\mathcal F$-acyclic.
\end{proof}

\begin{lemma}\label{lem:open}
Let $\{V_i\}_{i\in I}$ be a Zariski open covering of $X$.  We assume that $\mathcal C(\mathcal F, X/S)$ is non-empty.
An element $C$ of $\mathcal C(\mathcal F, X/S)$ is minimal if and only if, for any $i\in I$, $C_{V_i}=C\times_SV_i$ is a minimal element of $\mathcal C(\mathcal F, V_i/S)$.
\end{lemma}
\begin{proof}
We firstly prove the "only if" part. Let $C\in \mathcal C^{\mathrm{min}}(\mathcal F, X/S)$ and $V$ an open subscheme of $X$. It is easy to see that $C_V\in \mathcal C(\mathcal F|_V, V/S)$. Let $C^\prime$ be a minimal element of $\mathcal C(\mathcal F|_V, V/S)$ contained in $C_V$. Let $ \overline{C^\prime}$ be the closure of $C^\prime$ in $\ct(X/S)$. Then $\mathcal F$ is micro-supported on $(C\times_X(X-V))\bigcup \overline{C^\prime}$ relative to $S$.
Since $(C\times_X(X-V))\bigcup \overline{C^\prime}\subseteq C$, we have $C= (C\times_X(X-V))\bigcup \overline{C^\prime}$ since  $C$ is minimal. Hence $C_V= C^\prime$ is also minimal. For the "if" part, we suppose $C''$ is a minimal element of $\mathcal C(\mathcal F,X/S)$ contained in $C$. By the "only if" part, for any $i\in I$, the restriction $C''_{V_i}=C''\times_SV_i$ is a minimal element of $\mathcal C(\mathcal F|_{V_i}, V_i/S)$. Since $C_{V_i}$ is also minimal, we have $C''_{V_i}=C_{V_i}$, i.e., $C''=C$.

\end{proof}

\begin{lemma}[{cf. \cite[Lemma 2.1]{Beilinson15}}]\label{bei2.1}
\ \
\begin{itemize}
\item[(i)]We assume that $\mathcal C(\mathcal F,X/S)$ is nonempty. Then the base of $SS^w(\mathcal F, X/S)$ equals the support of $\mathcal F$. Let $C'$ be an element of $\mathcal C^{\mathrm{min}}(\mathcal F, X/S)$. Then the base of $C'$ equals the support of $\mathcal F$.
\item[(ii)]Assume that $\mathcal F$ is micro-supported on $C$ relative to $S$. Then, for every test pair $(g,h):Y\leftarrow U\rightarrow X$ relative to $S$ which is $C$-transversal relative to $S$, the map $g$ is universally locally acyclic with respect to $h^\ast\mathcal F$.
\item[(iii)]
The complex $\mathcal F$ is micro-supported on the zero section of $\mathbb T^\ast (X/S)$ if and only if each cohomology of $\mathcal F$ is locally constant on $X$.
\item[(iv)]
Let $\cF'$ and $\cF''$ be objects of $D^b_c(X,\Lambda)$ and $\cF'\rightarrow \cF\rightarrow \cF'\rightarrow$ a distinguished triangle. If $\cF'$ (resp. $\cF''$) is micro-supported on  $C'\subseteq \bT^*(X/S)$ (resp. $C''$) relative to $S$, then $\cF$ is micro-supported on $C'\cup C''$ relative to $S$.
\end{itemize}
\end{lemma}

\begin{proof}
 (i) Let $Z$ be the support of $\mathcal F$, $B$ the base of $SS^w(\mathcal F, X/S)$ and $B'$ the base of $C'$. We have $SS^w(\mathcal F, X/S)\subseteq C'$ and $B\subseteq B'$.
   By Lemma \ref{lem:open}, we have $C'\times_X(X-Z)=\emptyset$ since $\mathcal F|_{X-Z}=0$. Hence $B\subseteq B'\subseteq Z$. We need to show that $Z\subseteq B$.
   Replacing $X$ by $X-B$, we are reduced to show that $SS^\omega(\mathcal F, X/S)$ is non-empty if $\mathcal F\neq 0$. Consider the test pair $(r,\mathrm{id}):\mathbb A^1_S\xleftarrow{r} X\xrightarrow{\mathrm{id}} X$, where $r:X\rightarrow \mathbb A^1_S$ is a composition of $f:X\rightarrow S$ and the zero-section $i:S\rightarrow \mathbb A^1_S$. It is not $\mathcal F$-acyclic if $\mathcal F\neq 0$.

(ii) We need to prove that, after any base change $r\colon Y'\rightarrow Y$, the morphism $g':U'=U\times_SS'\rightarrow Y'$ is locally acyclic with respect to $h'^*\mathcal F$, where $h':U'\rightarrow X$ denotes the composition of the canonical morphism $r':U'\rightarrow U$ and $h:U\rightarrow X$. By d\'evissage, we only need to treat the case where $r\colon Y'\rightarrow Y$ is smooth.  We can replace $U$ by a Zariski neighborhood of the support of $h^*\mathcal F$.  By (i), the support of $\mathcal F$ is contained in the base of $C$. Hence, we may assume that $h$ is smooth by Lemma \ref{lem:1}. Then, the composition of the canonical maps $r':U'\rightarrow U$ and $h:U\rightarrow X$ is smooth, which implies that $h'=h\circ r'$ is $C$-transversal relative to $S$ (Lemma \ref{lem:1}). Since $r:Y'\rightarrow Y$ is smooth, the morphism $g:U\rightarrow Y$ is $h^{\circ}C$-transversal relative to $S$ implies that $g:U'\rightarrow Y'$ is $h'^{\circ}C$-transversal relative to $S$. Hence, $(g',h')$ is also a test pair of $X$ which is $C$-transversal relative to $S$. We deduce that $g'$ is locally acyclic with respect to $h'^{*}\mathcal F$.

(iii) If $\mathcal F$ is micro-supported on $\mathbb T_X^*(X/S)$ relative to $S$, then the test pair $({\mathrm{id}},{\mathrm{id}}):X\leftarrow X\rightarrow X$ is $\mathbb T_X^*(X/S)$-transversal. Hence $\mathrm{id}\colon X\rightarrow X$ is locally acyclic with respect to $\mathcal F$. This implies that each cohomology of $\mathcal F$ is a locally constant sheaf on $X$.

Conversely, If each cohomology of $\mathcal F$ is a locally constant sheaf on $X$, then any test pair $(g,h): Y\leftarrow U\rightarrow X$, where $g$ is smooth, is $\mathcal F$-acyclic. By Lemma \ref{lem:2}, the sheaf $\mathcal F$ is micro-supported on $\mathbb T_X^*(X/S)$ relative to $S$.

(iv) If a test pair $(g,h):Y\leftarrow U\rightarrow X$ is $(C'\cup C'')$-transversal relative to $S$, then $(g,h)$ is $\cF'$-acyclic (resp. $\cF''$-acyclic) relative to $S$, i.e., the morphism $g:U\rightarrow Y$ is  locally acyclic with respect to $h^*\cF'$ (resp. $h^*\cF''$). Since $\cF'\rightarrow \cF\rightarrow \cF'\rightarrow$ is a distinguished triangle, the morphism $g:U\rightarrow Y$ is also locally acyclic with respect to $h^*\cF$, i.e., the pair $(g,h)$ is $\cF$-acyclic relative to $S$.
\end{proof}


\begin{lemma}[{cf. \cite[Lemma 2.2]{Beilinson15}}]\label{bei2.2}
We assume that $C$ is an element in $\mathcal C(\mathcal F, X/S)$.
\begin{itemize}
\item[(1)]
Let $Y$ be an $S$-scheme smooth and of finite type over $S$ and $h:Y\rightarrow X$ an $S$-morphism. If $h$ is $C$-transversal relative to $S$, then, $h^{\circ}C$ is contained in $\mathcal C(h^*\mathcal F, Y/S)$.
\item[(2)]
Let $Y$ be an $S$-scheme smooth and of finite type over $S$ and $h:Y\rightarrow X$ an smooth and surjective $S$-morphism. If $h^{\circ}C$ is contained in $\mathcal C^{\mathrm {min}}(h^*\mathcal F, Y/S)$, then $C$ is contained in $\mathcal C^{\mathrm {min}}(\cF, X/S)$.
\item[(3)] Let $Z$ be an $S$-scheme smooth and of finite type over $S$ and $g:X\rightarrow Z$ an $S$-morphism. If the base $B(C)$ is proper over $Z$, then $g_{\circ}C$ is contained in $\mathcal C(\mathrm R g_*\mathcal F, Z/S)$.
\end{itemize}
\end{lemma}
\begin{proof}
(1) Let $U$ be a $S$-scheme smooth and of finite type over $S$ and $r:U\rightarrow Y$ an $S$-morphism which is $h^{\circ}C$-transversal relative to $S$. Since $h:Y\rightarrow X$ is $C$-transveral relative to $S$, the composition $h\circ r:U\rightarrow X$ is $C$-transversal relative to $S$. Hence, a test pair $(g,r): Z\leftarrow U\rightarrow Y$ relative to $S$  is $h^{\circ}C$-transversal relative to $S$ implies that the pair $(g,h\circ r):Z\leftarrow U\rightarrow X$ is $C$-transversal relative to $S$. Since $C$ is an element of $\mathcal C(\mathcal F,X/S)$, the latter implies that $g$ is locally acyclic with respect to $ r^*(h^*\mathcal F)$, i.e., the pair $(g,r)$ is $\mathcal F$-acyclic. Hence, $h^{\circ}C$ is contained in $\mathcal C (F, X/S)$.

(2) Let $C'$ be an element in $\mathcal C^{\mathrm {min}}(\cF, X/S)$ that is contained in $C$. Then, by (1), $h^{\circ}C'$ is contained in $C(h^*\mathcal F, X/S)$. Since $h^{\circ}C$ is contained in $\mathcal C^{\mathrm {min}}(h^*\mathcal F, Y/S)$, the inclusion $h^{\circ}C'\subseteq h^{\circ}C$ is an identity. Since $h$ is smooth, we have $h^{\circ}C=C\times_XY$ and $h^{\circ}C'=C'\times_XY$ (Lemma \ref{lem:1}). Since $h$ is surjective, the equality $C\times_XY=C'\times_XY$ implies that $C=C'$.

(3) Let $(r,s):W\leftarrow V\rightarrow Z$ be a $g_{\circ}C$-transversal test pair relative to $S$. We put $X'=V\times_ZX$ and denote by $s':X'\rightarrow X$ and $g':X'\rightarrow V$ the canonical projections.  Since $s:V\rightarrow Z$ is $g_{\circ}C$-transversal relative to $S$, there exists a Zariski open neighborhood $U$ of $s^{-1}(B(g_\circ C))$ in $V$ which is smooth over $Z$. Hence $U'=U\times_ZX$ is a Zariski open neighborhood of $s'^{-1}(B(C))$ in $X'$ which is smooth over $X$. We have the following commutative diagram with Cartesian squares
\begin{equation*}
\xymatrix{\relax
U'\ar[d]_{g'|_{U'}}\ar[r]^{j'}\ar@{}|-(0.5){\Box}[rd]&X'\ar[r]^{s'}\ar[d]^{g'}\ar@{}|-(0.5){\Box}[rd]&X\ar[d]^g\\
U\ar[rd]_{r|_U}\ar[r]^j&V\ar[r]^s\ar[d]^r&Z\\
&W&}
\end{equation*}
By Lemma \ref{lem:1}(1), the restriction $s'|_{U'}:U'\rightarrow X$ is $C$-transversal relative to $S$. By definition, $r|_{U}:U\rightarrow W$ is $s^\circ (g_\circ C)$-transversal relative to $S$ implies that $r\circ g'|_{U'}:U'\rightarrow W$ is $s'^{\circ}C$-transversal relative to $S$. Hence, the test pair $(r\circ g'|_{U'},s'|_{U'}):W\leftarrow U'\rightarrow X$ is $C$-transversal relative to $S$. Since $C$ is an element of $\mathcal C(\cF,X/S)$, the map $r\circ g'|_{U'}$ is universally locally acyclic with respect to $(s'|_{U'})^*\mathcal F$ (Lemma \ref{bei2.1}(2)). Since $B(C)$ contains the support of $\mathcal F$ (Lemma \ref{bei2.1}(1)), the scheme $U'$ contains the support of $s'^*\mathcal F$ in $X'$. Therefore, $r\circ g'$ is universally locally acyclic with respect to $s'^*\mathcal F$. Since the support of $s'^*\mathcal F$, contained in $s'^{-1}(B(C))$ is also proper over $V$, the morphism $r:V\rightarrow W$ is locally acyclic with respect to $\mathrm Rg'_*(s'^*\mathcal F)$ by proper base change theorem. By the same theorem, we have $\mathrm Rg'_*(s'^*\mathcal F)\xrightarrow{\sim}s^*(\mathrm R g_*\mathcal F)$. Hence, the test pair $(r,s):W\leftarrow V\rightarrow Z$ is $\mathrm R g_*\mathcal F$-acyclic. In summary, each $g_{\circ}C$-transversal test pair relative to $S$ is $\mathrm R g_*\mathcal F$-acyclic, i.e., $\mathrm R g_*\mathcal F$ is micro-supported on $g_{\circ}C$ relative to $S$.

\end{proof}

\begin{proposition}[{cf. \cite[Lemma 2.5]{Beilinson15}}]\label{rssclosedimm}
Suppose that $S_{\mathrm{red}}$ is not the spectrum of a field. Let $P$ be an $S$-scheme smooth and of finite type over $S$ and $i:X\rightarrow P$ a closed immersion of $S$-schemes.  
\begin{itemize}
\item[(i)]
If $SS^w(i_*\mathcal F, P/S)$ and $SS^w(\mathcal F,X/S)$ exist, we have 
\begin{equation}\label{wsscloseimm}
SS^w(i_*\mathcal F,P/S)=i_{\circ}SS^w(\mathcal F,X/S).
\end{equation}
\item[(ii)]
If  $SS(i_*\mathcal F, P/S)$ exists, then $SS(\mathcal F,X/S)$ also exists. Moreover, after replacing $S$ by a Zariski open dense subscheme, we have
\begin{equation}\label{sscloseimm}
SS(i_*\mathcal F,P/S)=i_{\circ}SS(\mathcal F,X/S).
\end{equation}
\end{itemize}
\end{proposition}

Please confer to \cite[Lemma 2.5]{Beilinson15} for the case where $S_{\mathrm{red}}=\mathrm{Spec}(k)$ is the spectrum of a field.  We mimic the part of {\it loc. cit.} where $k$ has infinite many elements to prove Proposition \ref{rssclosedimm}.

\begin{proof}
(i). Since $i:X\rightarrow P$ is a closed immersion, a weak test pair $(g,h):\mathbb A^1_S\leftarrow U\rightarrow P$ relative to $S$ is $i_*\mathcal F$-acyclic if and only if $(g',h'):\mathbb A^1_S\leftarrow U\times_PX\rightarrow X$ is $\mathcal F$-acyclic, where $g':U\times_PX\rightarrow Y$ is the composition of the first projection $\mathrm{pr}_1:U\times_PX\rightarrow U$ and $g:U\rightarrow Y$ and $h':U\times_PX\rightarrow X$ is the base change of $h:U\rightarrow P$. It implies that $i_{\circ}SS^w(\mathcal F,X/S)$ is an element of $\mathcal C^w(i_*\mathcal F,P/S)$, i.e., $SS^w(i_*\mathcal F,P/S)\subseteq i_{\circ}SS^w(\mathcal F,X/S)$.

Let $x$ be a point of $X$, $\alpha$ a non-zero vector of $\mathbb T^*_x(X/S)$ and $\beta$ a pre-image of $\alpha$ in $\mathbb T^*_x(P/S)$.
 Let $(u,v):\mathbb A^1_S\leftarrow V\rightarrow X$ be a weak test pair of $X$ relative to $S$, such that $V$ is a Zariski open neighborhood of $x$ and that $du_x=\alpha$. Since $P$ and $X$ are smooth over $S$, we can find elements $t_1,\cdots,t_m\in \mathcal O_{P,x}$ such that
\begin{itemize}
\item[1.]
$\mathcal O_{X,x}=\mathcal O_{P,y}/(t_{n+1},\cdots,t_m)$ $(n<m)$, and the element $a\in \mathcal O_{X,x}$ induced by $u:V\rightarrow \mathbb A^1_S$ writes
$a=\sum_{1\leq j\leq n}s_i\bar t_i$ $(s_i\in \mathcal O_S)$;
\item[2.]
$dt_1\otimes 1,\cdots,dt_m\otimes 1$ is a base of $k(x)$-vector space $\mathbb T^*_x(P/S)=\Omega^1_{P/S,x}\otimes_{\mathcal O_{P,x}}k(x)$ and $dt_1\otimes 1,\cdots,dt_n\otimes 1$ is a base of $k(x)$-vector space $\mathbb T^*_x(X/S)=\Omega^1_{X/S,x}\otimes_{\mathcal O_{X,x}}k(x)$.
\end{itemize}
Notice that $da\otimes 1=\alpha\in \bT^*_x(X/S)$. We can choose an element
$b=\sum_{1\leq j\leq n}s_jt_j+\sum_{n+1\leq j\leq m}r_jt_j\in \mathcal O_{P,x}$ $(r_j\in \mathcal O_S)$, such that $db\otimes 1=\beta$. Observe that $\bar b=a\in \mathcal O_{X,x}$. The element $b$ induces a weak test pair $(u',v'): \mathbb A^1_S\leftarrow V'\rightarrow P$, where $V'$ is a Zariski neighborhood of $x\in P$ and the restriction of $u'$ to $V'\times_PX$ coincides with $u:V\rightarrow \mathbb A^1_S$ in a neighborhood of $x\in X$. This fact implies that, if $\alpha\in SS_x^w(\mathcal F,X/S)$, any vector $\beta\in di^{-1}_x(\alpha)$ is contained in $SS_x^w(i_*\mathcal F,P/S)$. Hence, we get \eqref{wsscloseimm}.

(ii). By Lemma \ref{lem:open}, this is a local problem for the Zariski topology on $X$. We may assume that $P$ is affine over $S$ and that there exists an \'etale map $\chi:P\rightarrow \mathbb A^m_S$. Since $X$ is also smooth over $S$, there is a canonical injection $i_0:\mathbb A^n_S\rightarrow \mathbb A^m_S=\mathbb A^n_S\times_S\mathbb A^{n-m}_S$ induced by the zero-section of $\mathrm A^{m-n}_S$, such that the following diagram is Cartesian
\begin{equation}\label{cartdiagxp}
\xymatrix{\relax
X\ar[r]^i\ar[d]_{\chi_0}\ar@{}|-(0.5){\Box}[rd]&P\ar[d]^{\chi}\\
\mathbb A^n_S\ar[r]_-(0.5){i_0}&\mathbb A^m_S}
\end{equation}
The canonical projection $\mathrm{pr}:\mathbb A^m_S=\mathbb A_S^n\times_S \mathbb A^{m-n}_S\rightarrow \mathbb A_S^n$
induces a map $\chi_n=\mathrm{pr}\circ \chi:P\rightarrow \mathbb A^n_S$ and also a section $s\colon \mathbb T^\ast (X/S)\hookrightarrow \mathbb T^\ast (P/S)\times_P X$ of the canonical map $di\colon  \mathbb T^\ast (P/S)\times_P X\rightarrow  \mathbb T^\ast (X/S)$.
Consider the following commutative diagram with Cartesian squares
 \begin{equation*}
 \xymatrix{\relax
 X\ar[r]^-(0.5){\delta}&X\times_{\mathbb A^n_S}X\ar@{}|{\Box}[rd]\ar[d]_{\mathrm{pr}'_1}\ar[r]^-(0.5){i\times \mathrm{id}}&P\times_{\mathbb A^n_S}X\ar@{}|{\Box}[rd]\ar[d]_{\mathrm{pr}_1}\ar[r]^-(0.5){\mathrm{pr}_2}&X\ar[d]^{\chi_0}\\
 &X\ar[r]_i&P\ar[r]_-(0.5){\chi_n}&\mathbb A^n_S}
 \end{equation*}
Since $\chi_0:X\rightarrow \mathbb A^n_S$ is \'etale, the scheme $X\times_{\mathbb A^n_S}X$ is a disjoint union of the diagonal $\delta(X)$ and its complement $Y$. We denote by $\widetilde P$ the complement of $Y$ in $P\times_{\mathbb A^n_S}X$ and by $\rho:\widetilde P\rightarrow X$ and $\phi:\widetilde P\rightarrow P$ the restriction of $\mathrm{pr}_1$ and $\mathrm{pr}_2$ on $\widetilde P$, respectively.
We observe that $X=\widetilde P\times_PX$, which gives a section $\tilde i: X\rightarrow \widetilde P$ of $\rho:\widetilde P\rightarrow X$.  By Lemma \ref{bei2.1}(i), the base of $SS(i_\ast\mathcal F, P/S)$ is contained in $X$. Hence, we have a closed conical subset $\rho_\circ \phi^\circ SS(i_\ast\mathcal F, P/S)=s^{-1}(SS(i_\ast\mathcal F, P/S))$ of $\mathbb T^\ast (X/S)$ and we denote it by $C(\mathcal F)$. Since $\mathcal F=\mathrm R\rho_\ast\phi^\ast(i_\ast\mathcal F)$, we have $C(\mathcal F)\in\mathcal C(\mathcal F, X/S)$ (Lemma
\ref{bei2.2}). For any $C'\in\mathcal C(\mathcal F, X/S)$, we have $i_\circ C'\in\mathcal C(i_\ast\mathcal F, P/S)$ ({\it loc. cit.}). Hence $SS(i_\ast\mathcal F, P/S)\subseteq i_\circ C'$ and
\begin{equation*}
C(\mathcal F)=\rho_\circ \phi^\circ SS(i_\ast\mathcal F, P/S)\subseteq \rho_\circ \phi^\circ(i_\circ C')=C'
\end{equation*}
Thus, $C(\mathcal F)$ is the smallest element of $\mathcal C(\mathcal F, X/S)$, i.e.,  $SS(\mathcal F, X/S)=C(\mathcal F)$. By {\it loc. cit.}, we have $SS(i_*\cF,P/S)\subseteq i_\circ SS(\cF, X/S)$.

Let $\theta:\mathbb A^m_S\rightarrow A^m_S$ be an $S$-automorphism of  $\mathbb A^m_S$ that satisfies $\theta\circ i_0=i_0$. We put $\chi_{\theta}=\theta\circ\chi$. The composition $\mathrm{pr}\circ\chi_{\theta}:P\rightarrow \mathbb A^n_S$ induces a section $s_{\theta}:\mathbb T^*( X/S)\rightarrow \mathbb T^*(P/S)\times_PX$ of $di:\mathbb T^*(X/S)\rightarrow \mathbb T^*(P/S)\times_PX$. Following the same argument above, we have
\begin{equation}\label{invofssx}
s_{\theta}^{-1}(SS(i_*\mathcal F,P/S))=SS(\mathcal F, X/S)=s^{-1}(SS(i_*\mathcal F,P/S)).
\end{equation}
After shrinking $S$, we may assume that $S$ is integral and we denote by $\eta$ the generic point of $S$. To prove \eqref{sscloseimm}, it suffices to show
\begin{equation}\label{closeimmgene}
SS(i_*\mathcal F,P/S)\times_S\eta=i_{\circ}SS(\mathcal F,X/S)\times_S\eta
\end{equation}
The case where $\mathrm{supp}(\mathcal F)\cap X_{\eta}=\emptyset$ is trivial and we consider the case where $\mathrm{supp}(\mathcal F)\cap X_{\eta}\neq\emptyset$. Since
\begin{equation*}
i_{\circ}SS^w(\cF,X/S)=SS^w(i_*\mathcal F,P/S)\subseteq SS(i_*\mathcal F,P/S),
\end{equation*}
the conical
set $SS(i_*\mathcal F,P/S)$ contains the restriction of the conormal bundle $\mathbb N_{X|P}$ to the support of $i_*\mathcal F$. If \eqref{closeimmgene} is false, we can fine a closed point $z\in X_{\eta}$ and a non-zero vector $\lambda\in \mathbb T^*_z(X/S)$, such that $di_z^{-1}(\lambda)$ contains vectors in $ (i_{\circ}SS(\mathcal F,X/S))_z\backslash (SS(i_*\mathcal F,P/S))_z$.
Replacing $S$ by a Zariski open dense subset, we can take a sufficiently general $\chi_{\theta}:P\rightarrow \mathbb A^m_S$, such that $\lambda\notin (s^{-1}_{\theta}(SS(i_*\mathcal F,P/S)))_z$. It contradicts to \eqref{invofssx}.
\end{proof}

\section{Existence of relative singular supports}
In this section, let $S$ be a connected Noetherian $\bZ[1/\ell]$-scheme. We take the notation and assumptions of \ref{projspace} and \ref{legtran}. 

\begin{lemma}[{cf. \cite[Lemma 3.3]{Beilinson15}}]\label{bei.3.3}
Let $C$ be a closed conical subset of $\bT^*(\bP/S)$ and $\cF$ an object of $D^b_c(\bP,\Lambda)$. If $\cF$ is micro-supported on $C^+$ relative to $S$, then the Radon transform $R_S(\cF)$ is micro-supported on $C^{\vee+}$ relative to $S$. Conversely, if $R_S(\cF)$ is micro-supported on $C^{\vee+}$ relative to $S$, then, after replacing $S$ by a Zariski open dense subset, $\cF$ is micro-supported on $C^+$ relative to $S$.
\end{lemma}
\begin{proof}
If $\cF$ is micro-supported on $C^+$ relative to $S$, then the Radon transform $R_S(\cF)$ is micro-supported on $\rpr^{\vee}_{\circ}\rpr^{\circ}(C^+)=C^{\vee+}$ relative to $S$ (Lemma \ref{bei2.2}).

Conversely, if  $R_S(\cF)$ is micro-supported on $C^{\vee+}$, then the complex $(R^{\vee}_S\circ R_S)(\cF)$ is micro-supported on $\rpr_{\circ}\rpr^{\vee\circ}(C^{\vee+})=C^{\vee\vee+}=C^+$.
After shrinking $S$, we may assume that the mapping cone of the adjunction $\cF\rightarrow (R^{\vee}_S\circ R_S)(\cF)$ has locally constant cohomologies (Proposition \ref{bei1.6.1}). By Lemma \ref{bei2.1}, $\mathcal F$ is also micro-supported on $C^+$ relative to $S$.
\end{proof}

\begin{lemma}[cf. {\cite[Lemma 3.4]{Beilinson15}}]\label{bei3.4}
Let $(g,h):Y\leftarrow U\rightarrow \bP$ be a test pair of $\bP$ relative to $S$. We have the following diagram with Cartesian squares (\ref{legtran})
\begin{equation}\label{dbei3.4}
\xymatrix{\bH_U\ar[r]^{h_U}\ar[d]_{\rpr_U}\ar@{}[rd]|-{\Box}&\bH\ar[r]^{\rpr^{\vee}}\ar[d]^{\rpr}&\bP^{\vee}\\
U\ar[r]_h\ar[d]_g&\bP&&\\
Y&&}
\end{equation}
Let $C$ be a strict closed conical subset of $\bT^*(\bP/S)$ and  $E\subseteq \bH$ the image of $\rP(C)$ in $\bH$ by Legendre transform. Then, the test pair $(g,h)$ is $C$-transversal relative to $S$ if and only if the test pair $(g\circ\rpr_U,\rpr^{\vee}\circ h_U)$ of $\bP^{\vee}$ is $\bT^*\bP^{\vee}$-transversal relative to $S$ at $E_U=E\times_\bP U$.
\end{lemma}
In fact, Lemma \ref{bei3.4} is deduced by \cite[Lemma 3.4]{Beilinson15} by taking fibers relative to $S$.

\begin{theorem}[{cf. \cite[Theorem 3.2]{Beilinson15}}]\label{exrssproj}
Let $\cF$ be an object of $D^b_c(\bP,\Lambda)$.  Then, after replacing $S$ by a Zariski open dense subset, the relative singular support $SS(\cF,\bP/S)$ exists and we have
\begin{equation}\label{eprequss}
\rP(SS(\cF,\bP/S))=E_{\rpr}(\rpr^{\vee*}R_S(\cF))\subseteq \bH
\end{equation}
by Legendre transform (\ref{legtran}). Assume that $S$ is irreducible and let $\eta$ be the generic point of $S$. If $\mathrm{supp}(\cF|_{\bP_{\eta}})\neq \bP_{\eta}$, then $SS(\cF,\bP/S)$ is the conical subset of $\bT^*(\bP/S)$ associated to $E_{\rpr}(\rpr^{\vee*}R_S(\cF))$. If, after shrinking $S$, we have $\mathrm{supp}(\mathcal F)=\bP$, then $SS(\cF,\bP/S)$ is the union of conical subset of $\bT^*(\bP/S)$ associated to $E_{\rpr}(\rpr^{\vee*}R_S(\cF))$ and the zero-section of $\bT^*(\bP/S)$ after shrinking $S$ again.
\end{theorem}
\begin{proof}
After shrinking $S$, we may assume that $\pi:\bP\rightarrow S$ is universally locally acyclic with respect to $\cF$ and that the mapping cone of the adjunction $\cF\rightarrow (R^{\vee}_S\circ R_S)(\cF)$ has locally constant cohomologies (Proposition \ref{bei1.6.1}).  By smooth base change theorem, the composition $\pi^{\vee}\circ \rpr:\bH\rightarrow S$ is universally locally acyclic with respect to $\rpr^*\cF$. By proper base change theorem, the morphism $\pi^{\vee}:\bP^{\vee}\rightarrow S$ is universally locally acyclic with respect to $\cG=R_S(\cF)$. In the following, we replace $\cF$ by $R^{\vee}_S(\cG)$. We simply put $E=E_{\rpr}(\rpr^{\vee*}R_S(\cF))\subseteq \bH$ and we denote by $C\subseteq \bT^*(\bP/S)$ the strict closed conical set associated to $E$ by Legendre transform. Let $C'$ be an element of $C^{\mathrm{min}}(\mathcal F, \bP/S)$ and $E' \subseteq\bH$ the Legendre transform of $\rP(C')$.
To prove the theorem, it is enough to show that $C^+=C'^+$.

By the equality \eqref{dpreqn}, the test pair $(\rpr, \rpr^{\vee}):\bP\leftarrow\bH-E'\rightarrow \bP^{\vee}$ of $\bP^{\vee}$ is
$C'^{\vee +}$-transversal relative to $S$. By Lemma \ref{bei.3.3}, the complex $\cF$ is micro-supported on $C'$ relative to $S$ implies that the complex $\cG$ is micro-supported on $C'^{\vee +}$ relative to $S$. Hence, the morphism $\rpr:\bH-E'\rightarrow \bP$ is universally locally acyclic with respect to $\rpr^{\vee*}\cG$, i.e., we have $E\subseteq E'$ and $C^+\subseteq C'^+$.

To show $C'^+\subseteq C^+$, it suffices to verify that $C^+$ is an element of $\cC(\cF,\bP/S)$.
Let $(g,h):Y\leftarrow U\rightarrow \bP$ be a test pair of $\bP$ relative to $S$. We assume that $(g,h)$ is $C^+$-transversal relative to $S$. We take the notation of diagram \eqref{dbei3.4} and we put $E_U=E\times_{\bP}U$. It is sufficient to show that $g:U\rightarrow Y$ is locally acyclic with respect to $h^*\cF=\rR\rpr_{U*}(h^*_U(\rpr^{\vee *}\cG))$. Since $\rpr_U:\bH_U\rightarrow U$ is proper, by proper base change theorem, we are reduced to prove that the morphism $g\circ\rpr_U:\bH_U\rightarrow Y$ is locally acyclic with respect to $h^*_U(\rpr^{\vee *}\cG)$. Indeed,
\begin{itemize}
\item[(i)]
By the definition of $E$, the morphism $\rpr:\bH-E\rightarrow \bP$ is universally locally acyclic with respect to $\rpr^{\vee *}\cG$. Hence $\rpr_U:\bH_U-E_U\rightarrow U$ is universally locally acyclic with respect to $h^*_U(\rpr^{\vee *}\cG)$. Since $(g,h)$ is $C^+$-transversal relative to $S$ and $C^+$ contains the zero-section of $\bT^*(\bP/S)$, the morphism $g:U\rightarrow Y$ is smooth. By \cite[2.14]{SGA4.5}, the morphism $g\circ\rpr_U:\bH_U\rightarrow Y$ is locally acyclic with respect to $h^*_U(\rpr^{\vee *}\cG)$ on $\bH_U-E_U$.
\item[(ii)]
By Lemma \ref{bei3.4}, the test pair $(g\circ\rpr_U,\rpr^{\vee}\circ h_U)$ of $\bP^{\vee}$ is $\bT^*\bP^{\vee}$-transversal relative to $S$ at $E_U=E\times_\bP U$. Notice that $\pi^{\vee}:\bP^{\vee}\rightarrow S$ is universally locally acyclic with respect to $\cG$. By Proposition \ref{lem:nonempty}, the morphism $g\circ \rpr_U:\bH_U\rightarrow Y$ is locally acyclic with respect to $h^*_U(\rpr^{\vee *}\cG)$ in a neighborhood of $E_U$.
\end{itemize}
\end{proof}

\begin{theorem}[{cf. \cite[Theorem 1.3]{Beilinson15}}]\label{existrss}
Let $f:X\rightarrow S$ be a smooth morphism of finite type and $\mathcal F$ an object of $D^b_c(X,\Lambda)$. Then, after replacing $S$ by a Zariski open dense subscheme, the singular support $SS(\mathcal F,X/S)$ of $\mathcal F$ relative to $S$ exists.
\end{theorem}

\begin{proof}
The case where $S_{\mathrm{red}}$ is the spectrum of a field is done by \cite[Theorem 1.3]{Beilinson15}.
It suffices to consider the case where $S_{\mathrm{red}}$ is not the spectrum of a field. 
 By Lemma \ref{lem:open}, this is a local problem for the Zariski topology of $X$. We may assume that $f:X\rightarrow S$ is affine. By Proposition \ref{rssclosedimm}, we are reduced to the case where $X=\mathbb A^n_S$. By Lemma \ref{lem:open} again, we are reduced to the case where $X=\bP^n_S$. It is deduced by Theorem \ref{exrssproj}.
 \end{proof}

\begin{proposition}[{\cite[Theorem 3.7]{Beilinson15}}]\label{bei3.7}
Suppose that $S$ is an irreducible excellent Noetherian scheme. Let $\cG$ be an object of $D^b_c(\bP^{\vee},\Lambda)$. Then, after replacing $S$ by a generically dimensional Zariski open dense subscheme, we have $\dim E_{\rpr}(\rpr^{\vee*}\cG)\leq\dim\bP-1$.
\end{proposition}

The origin proof of \cite[Theorem 3.7]{Beilinson15} is for the case where $S$ is an irreducible algebraic variety over a field. In fact, the same proof is also valid in our case.

\begin{corollary}\label{ssxydimx}
Suppose that $S$ is an irreducible excellent Noetherian scheme.  
Let $\cF$ be an object of $D^b_c(\bP,\Lambda)$.
Then, after replacing $S$ by a generically dimensional Zariski open dense subscheme,  we have
\begin{equation}\label{dimssleqn}
  \dim SS(\cF,\bP/S)\leq \dim \bP.
\end{equation}
\end{corollary}
\begin{proof}
 By Proposition \ref{bei3.7} (cf. \cite[Theorem 3.7]{Beilinson15}), after replacing $S$ by a generically dimensional Zariski open dense subscheme, we have
  \begin{equation*}
   \dim_k E_{\rpr}(\rpr^{\vee*}R_S(\cF))\leq \dim_k\bP-1,
  \end{equation*}
  i.e., $\dim_k\rP(SS(\cF,\bP/S))\leq \dim_k\bP-1$ (Theorem \ref{exrssproj}). Hence we obtain \eqref{dimssleqn}.
\end{proof}

\begin{proposition}[{cf. \cite[Theorem 1.7]{Beilinson15}}]\label{bei1.7}
Suppose that $S$ is an integral excellent Noetherian scheme and that the Veronese embedding $\wt i:\bP\rightarrow\wt\bP$ in \ref{projspace} has degree $d\geq 3$.  Let $\mathcal F$ be an object of $D^b_c(\bP,\Lambda)$ and we put $D=NC(\widetilde{R}_S(\wt i_\ast \cf), {\widetilde{\bP}}^\vee)$ (\ref{ncl}).
Then, after replacing $S$ by a Zariski open dense subset, we have
 \begin{itemize}
    \item[(i)]
For each irreducible component $D_\alpha$ ($\alpha\in I$) of $D$,
there is a unique irreducible closed conical subset
$C_{\alpha}\subseteq \bT^*(\bP/S)$
such that $D_{\alpha}=\wt\rpr^\vee(\rP(\wt i_\circ C_{\alpha}))$ and that
the projection $\wt\rpr^\vee:\bP(\wt i_\circ C_{\alpha})\rightarrow D_{\alpha}$ is generically radicial. We have $SS(\mathcal F, \bP/S)=\bigcup_{\alpha\in I}C_{\alpha}$.
 \item[(ii)]
   The closed subscheme $D$ is an effective Cartier divisor of $\wt\bP^\vee$ relative to $S$.
 \end{itemize}
\end{proposition}
\begin{proof}
By Theorem \ref{exrssproj}, we may assume that $SS(\cF,\bP/S)$ and $SS(\wt R_S(\wt i_*\cF),\wt\bP/S)$ exist. We simply put $C=SS(\cF,\bP/S)$ and put $\wt C=SS(\wt R_S(\wt i_*\cF),\wt\bP^\vee/S)$. By Lemma \ref{bei2.1}(iii), we have $D=B(\wt C)$. After shrinking $S$ again, we have $\wt C=(\wt i_\circ C)^\vee$ or $\wt C=(\wt i_\circ C)^{\vee +}$ (Proposition \ref{rssclosedimm}(ii)  and Lemma \ref{bei.3.3}). Hence, we get $D=B(\wt C)=\wt\rpr^\vee(\rP(\wt i_\circ C))$. We denote by $\{C_\alpha\}_{\alpha\in I}$ the set of irreducible components of $C$.
By Lemma \ref{gendiml}, we may assume that $S$ and each irreducible component of $\dim \rP(\wt i_\circ C)$ are generically dimensional. By Corollary \ref{ssxydimx},  we have $\dim \rP(\wt i_\circ C)\leq \dim\wt\bP^\vee-1$. Hence, after shrinking $S$, the map \begin{equation*}
\wt\rpr^\vee:\wt\bH\rightarrow\wt\bP^\vee,\ \ \ \rP(\wt i_\circ C_\alpha)\mapsto D_\alpha=\wt\rpr^\vee(\rP(\wt i_\circ C_\alpha))
\end{equation*}
induces a one-to-one correspondence between the sets of irreducible components of $C$ and $D$, and moreover, $\wt\rpr^\vee:\rP(\wt i_\circ C_\alpha)\rightarrow D_\alpha$ is generically small (Proposition \ref{geomctod}).

For any $s\in S$, we denote by $\wt i_s:\bP_s\rightarrow\wt\bP_s$ (resp. $\wt\rpr_s:\wt\bH_s^\vee\rightarrow \wt\bP_s$ and $\wt\rpr_s^\vee:\wt\bH_s^\vee\rightarrow \wt\bP_s^\vee$) the fiber of $\wt i:\bP\rightarrow\wt\bP$ (resp. $\wt\rpr:\wt\bH^\vee\rightarrow \wt\bP$ and $\wt\rpr^\vee:\wt\bH^\vee\rightarrow \wt\bP^\vee$) at $s$, and by $\wt R_s$ the Radon transform for the pair $(\wt\rpr_s,\wt\rpr_s^\vee):\wt\bP_s\leftarrow\wt
\bH_s\rightarrow\wt\bP^\vee_s$. By proper base change theorem, we have the canonical isomorphism
\begin{equation*}
(\wt R_S(\wt i_*\cF))|_{\wt\bP^\vee_s}\xrightarrow{\sim} \wt R_s(\wt i_{s *}(\cF|_{\bP_s})).
\end{equation*}
Let $\eta$ denotes the generic point of $S$. By Lemma \ref{nclfiber}, we have
\begin{equation*}
D_\eta=NC((\wt R_S(\wt i_*\cF))|_{\wt\bP^\vee_\eta},\wt\bP^\vee_\eta)=NC(\wt R_\eta(\wt i_{\eta *}(\cF|_{\bP_\eta})),\wt\bP^\vee_\eta).
\end{equation*}
By \cite[Theorem 1.7]{Beilinson15}, $D_\eta=NC(\wt R_\eta(\wt i_{\eta *}(\cF|_{\bP_\eta})),\wt\bP^\vee_\eta)$ is an effective Cartier divisor of $\wt\bP^\vee_{\eta}$. Since $D$ is a reduced closed subscheme of $\wt\bP^\vee$ containing $D_\eta$, after shrinking $S$, the closed subscheme $D$ is an effective Cartier divisor of $\wt\bP^\vee$ relative to $S$.
\end{proof}

\begin{theorem}\label{sseqrssfiber}
Let $S$ be an excellent Noetherian scheme, $f:X\rightarrow S$ a smooth morphism of finite type and $\cF$ an object of $D^b_c(X,\Lambda)$. Then, after replacing $S$ by a Zariski open dense subset, for any $s\in S$, we have
 \begin{equation}\label{ssequalwant}
   SS(\cF|_{X_s})=(SS(\cF,X/S)\times_Ss)_\red.
 \end{equation}
\end{theorem}
\begin{proof}
By d\'evissage in the proof of Theorem \ref{existrss},  we are reduced to the case where $X=\bP$. We take the notation and assumptions of Proposition \ref{bei1.7} and we assume that $S$ is integral. We also omit the subscript "$\mathrm{red}$" in the proof for simplicity.
By Theorem \ref{exrssproj}, we may assume that $C=SS(\cF,\bP/S)$ exists. By \eqref{eprequss}, for any $s\in S$, we have $\rP(SS(\cF|_{\bP_s}))\subseteq \rP(C_s)$. After shrinking $S$, we may assume that $\supp(\cF)=\bP$ if and only if, for any $s\in S$, $\supp(\cF|_{\bP_s})=\bP_s$. Hence, we have $SS(\cF|_{\bP_s})\subseteq C_s$ for any $s\in S$. It implies that $\rP(\wt i_{s\circ}(SS(\cF|_{\bP_s})))$ is a closed subset of $\rP(\wt i_\circ C)_s$ for any $s\in S$.
By Lemma \ref{nclfiber}, after shrinking $S$,  for any $s\in S$, we have
\begin{equation*}
  D_s=NC((\wt R_S(\wt i_*\cF))|_{\wt\bP^\vee_s},\wt\bP^\vee_s)=NC(\wt R_s(\wt i_{s *}(\cF|_{\bP_s})),\wt\bP^\vee_s).
\end{equation*}
 By \cite[Thorem 1.7]{Beilinson15} (Proposition \ref{bei1.7} for $S=s$), the projection $\wt\rpr^\vee_s:\wt\bH_s\rightarrow\wt\bP^\vee_s$ induces a surjective map
 \begin{equation}\label{surjbei1.7}
 \wt\rpr^\vee_s:\rP(\wt i_{s\circ}(SS(\cF|_{\bP_s})))\twoheadrightarrow D_s=NC(\wt R_s(\wt i_{s *}(\cF|_{\bP_s})),\wt\bP^\vee_s).
 \end{equation}

After shrinking $S$, the projection $\wt\rpr^\vee:\wt\bH\rightarrow\wt\bP^\vee$ induces a one-to-one correspondence between the sets of irreducible components $\{\rP(\wt i_\circ(C_\alpha)\}_{\alpha\in I}$ and $\{D_\alpha\}_{\alpha\in I}$ and the canonical surjection $\wt\rpr^\vee:\rP(\wt i_\circ C_\alpha)\rightarrow D_\alpha$ is surjective and generically radicial (Proposition \ref{bei1.7}). Hence, after shrinking $S$, we may assume that (Lemma \ref{fibradicial})
 \begin{itemize}
 \item[(1)]
 For each $\alpha\in I$, the components $\rP(\wt i_\circ C_\alpha)$ and $D_\alpha$ are flat over $S$;
 \item[(2)]
For any $s\in S$, the projection $\wt\rpr^\vee_s:\wt\bH_s\rightarrow\wt\bP^\vee_s$ induces a one-to-one correspondence between the irreducible components of $\rP(\wt i_\circ C)_s$ and $D_s$. We denote by $\{(\rP(\wt i_\circ C)_s)_\beta\}_{\beta\in J(s)}$ (resp. $\{(D_s)_\beta\}_{\beta\in J(s)}$) the sets of irreducible components of $\rP(\wt i_\circ C)_s$ (resp. $D_s$). For each $\beta\in J(s)$, the canonical projection $\wt\rpr^\vee_s:\rP(\wt i_\circ C)_s)_\beta\rightarrow (D_s)_\beta$ is generically radicial.
 \end{itemize}

By \eqref{surjbei1.7} and (2), for any $s\in S$, the closed subset
$\rP(\wt i_{s\circ}(SS(\cF|_{\bP_s})))$ contains all generic points of irreducible components of $\rP(\wt i_\circ C)_s$. Hence, for any $s\in S$, we have
\begin{equation}\label{pssequal}
\rP(\wt i_{s\circ}(SS(\cF|_{\bP_s})))=\rP(\wt i_\circ C)_s.
\end{equation}
 Since both $\wt i_{s\circ}(SS(\cF|_{\bP_s}))$ and $\wt i_{s\circ} (C_s)$ are strict closed conical subsets of $\wt\bP_s$, \eqref{pssequal} implies that
 \begin{equation}\label{ssequal}
 \wt i_{s\circ}(SS(\cF|_{\bP_s}))=\wt i_{s\circ} (C_s).
 \end{equation}
 Since $d\wt i_{s}:\bT^*\wt\bP_s\times_{\wt\bP_s}\bP_s\rightarrow \bT^*\bP_s$ is a surjective map, \eqref{ssequal} implies that $SS(\cF|_{\bP_s})=C_s$.
 \end{proof}

By \cite[II, 6.9.1]{EGA4}, we have the following theorem. It implies that part (i) of Theorem \ref{introdellau} can be generalized to higher relative dimension cases if allowing to replace $S$ by a Zariski open dense subset.

\begin{theorem}\label{closedA}
Let $S$ be an excellent Noetherian scheme, $f:X\rightarrow S$ a smooth morphism of finite type and $\cF$ an object of $D^b_c(X,\Lambda)$. Then, after replacing $S$ by a Zariski open dense subset, there exists a closed subscheme $Z\subseteq \bT^*(X/S)$ which is flat over $S_{\mathrm{red}}$, such that, for any $s\in S$, we have
\begin{equation*}
 SS(\cF|_{X_s})=(Z_{s})_{\mathrm{red}}.
\end{equation*}
\end{theorem}

\begin{corollary}
Let $X$ and $S$ be irreducible excellent Noetherian schemes, $f:X\rightarrow S$ a smooth morphism of finite type and $\cF$ an object of $D^b_c(X,\Lambda)$.  Then, after replacing $S$ by a generically dimensional Zariski open dense subset, we have $X$ is generically dimensional, $SS(\cF,X/S)$ is equidimensional and has dimension $\dim X$.
\end{corollary}

\begin{proof}
It is deduced by Lemma \ref{gendiml}, Theorem \ref{sseqrssfiber} and \cite[Theorem 1.3]{Beilinson15}.
\end{proof}

\begin{corollary}[{cf. \cite[Theorem 1.5]{Beilinson15}}]
Let $S$ be an excellent Noetherian scheme, $f:X\rightarrow S$ a smooth morphism of finite type and $\cF$ an object of $D^b_c(X,\Lambda)$. Then, after replacing $S$ by a Zariski open dense subset, we have
\begin{equation*}
SS^w(\cF, X/S)=SS(\cF,X/S).
\end{equation*}
\end{corollary}
\begin{proof}
We may assume that $S$ is irreducible and let $\eta$ be the generic point of $S$.
We simply put $B=SS^w(\cF, X/S)$ and put $C=SS(\cF,X/S)$. It is sufficient to prove that $B_\eta=C_\eta$. A priori, we have $B\subseteq C$, hence $B_\eta\subseteq C_\eta$. Any $B_\eta$-transversal weak test pair $(g_0,h_0): \bA^1_\eta\leftarrow U_0\rightarrow X_\eta$ can be extended to a $B$-transversal week test pair $(g,h):\bA^1_S\leftarrow U\rightarrow X$ relative to $S$. Hence, any $B_\eta$-transversal weak test pair is $\cF|_{X_\eta}$-acyclic, i.e., $SS^w(\cF|_{X_\eta})\subseteq B_\eta$. By theorem \ref{sseqrssfiber}, we have $SS(\cF|_{X_\eta})=C_\eta$. By \cite[Theorem 1.5]{Beilinson15}, we have $SS^w(\cF|_{X_\eta})=SS(\cF|_{X_\eta})$. In summary, we have
\begin{equation*}
C_\eta=SS(\cF|_{X_\eta})=SS^w(\cF|_{X_\eta})\subseteq B_\eta\subseteq C_\eta.
\end{equation*}
\end{proof}

\begin{corollary}\label{cor:fex}
Let $S$ and $S'$ be excellent Noetherian schemes, $f:S'\rightarrow S$ a dominant morphism, $f:X\rightarrow S$ a smooth morphism of finite type, $X'=X\times_SS'$ and $\cF$ an object of $D^b_c(X,\Lambda)$. Then, after replacing $S$ and $S'$ by Zariski open dense subsets, we have
\begin{equation}\label{rsspullback}
SS(\cF|_{X'},X'/S')=(SS(\cF,X/S)\times_SS')_{\red}.
\end{equation}
\end{corollary}
\begin{proof}
We may assume that $S$ and $S'$ are irreducible. Let $\eta$ and $\eta'$ be the generic points of $S$ and $S'$, respectively. By \cite[Theorem 1.4]{Beilinson15} and Theorem \ref{sseqrssfiber}, we have
\begin{equation*}
SS(\cF|_{X'},X'/S')\times_{S'}\eta'=SS(\cF|_{\eta'})=(SS(\cF|_{\eta})\times_\eta\eta')_{\red}=(SS(\cF,X/S)\times_S{\eta'})_{\red}.
\end{equation*}
Then, after shrinking $S'$ again, we get \eqref{rsspullback}.
\end{proof}

\section{Generic constancy of characteristic cycles}

In this section, let $k$ be a perfect field of characteristic $p>0$ ($p\neq\ell$). All $k$-schemes are assumed to be of finite type over $\mathrm{Spec}(k)$.

\subsection{}
Let $X$ be a smooth $k$-scheme of equidimension $n$, $C$ a closed conical subset of $\bT^*X$ and $f:X\rightarrow \bA^1_k$ a $k$-morphism.
A closed point $v\in X$ is called {\it an at most $C$-isolated characteristic point of $f:X\rightarrow \bA^1_k$} if there is an open neighborhood $V\subseteq X$ of $v$ such that  $f: V-\{v\}\rightarrow \bA^1_k$ is $C$-transversal. A closed point $v\in X$ is called a {\it $C$-isolated characteristic point} if $v$ is an at most $C$-isolated characteristic point of $f:X\rightarrow \bA^1_k$ but  $f:X\rightarrow \bA^1_k$ is not $C$-transversal at $v$.

We assume that $C\subseteq \bT^*X$ is of equidimension $n$. Let $\{C_\alpha\}_{\alpha\in I}$ be the set of irreducible components of $C$. We denote by $o$ the origin of $\bA^1_k$. A base $\theta\in\bT^*_o\bA^1_k$ defines a section $\theta:X\rightarrow\bT^*\bA^1_k\times_{\bA^1_k}X$, hence a section $df\circ \theta:X\rightarrow \bT^*X$, where $df:\bT^*\bA^1_k\times_{\bA^1_k}X\rightarrow \bT^*X$ is the canonical morphism. Let $A=\sum_{\alpha\in I}m_\alpha[C_\alpha]$ $(m_\alpha \in\mathbb Z)$ be an $n$-cycle supported on $C$ and $v\in X$ an at most $C$-isolated characteristic point of $f:X\rightarrow \bA^1_k$. Then, there is an open neighborhood $V$  of $v$ in $X$ such that the intersection $(A,[(df\circ\theta)(W)])$ is a $0$-cycle supported on a closed point of $T^*_vV$. Its degree is independent of the choice of the base $\theta\in\bT^*_o\bA^1_k$ and we denote this intersection  number by $(A, df)_{\bT^*X,v}$.

\subsection{}(\cite[Theorem 5.9]{Saito15-2})
Let $X$ be a smooth $k$-scheme of equidimension $n$, $\cF$ an object of $D^b_c(X,\Lambda)$ and $\{C_\alpha\}_{\alpha\in I}$ the set of irreducible components of $SS(\cF)$.  Then, there exists a unique $n$-cycle $CC(\cF)=\sum_{\alpha\in I}m_\alpha [C_\alpha]$ $(m_\alpha\in \mathbb Z)$ of $\bT^*X$ supported on $SS(\cF)$,
satisfying the following Milnor type formula (\ref{eq:milnor}):

For any \'etale morphsim $g:V\rightarrow X$, any morphism $f:V\rightarrow\bA^1_k$, any isolated $g^\circ SS(\cF)$-characteristic point $v\in V$ of $f:V\rightarrow\bA^1_k$ and a geometric point $\bar v$ of $V$ above $v$, we have
\begin{equation}\label{eq:milnor}
  -\sum_i(-1)^{i}\dimtot (\rR^i\Phi_{\bar v}(g^\ast \mathcal F, f))=(g^*CC(\cF),df)_{\bT^*V,v},
\end{equation}
where $\rR\Phi_{\bar v}(g^*\cF,f)$ denotes the stalk at $\bar v$ of the vanishing cycle complex of $g^*\cF$ relative to $f$, $\dimtot(\rR^i\Phi_{\bar v}(g^*\cF,f))$ the total dimension of  $\rR^i\Phi_{\bar v}(g^*\cF,f)$ (\ref{defdimtot}) and $g^*CC(\cF)$ the pull-back of $CC(\cF)$ to $\bT^*V$. We call $CC(\cF)$ the {\it characteristic cycle of} $\cF$.

\begin{proposition}[{\cite[Lemma 5.11, Lemma 5.13]{Saito15-2}}]\label{ccopenclose}
Let $X$ be a smooth $k$-scheme  of equidimension $n$ and $\cF$ be an object of $D^b_c(X,\Lambda)$. Then
\begin{itemize}
\item[(i)]
For any \'etale morphism $j:U\rightarrow X$, we have
\begin{equation*}
j^\ast CC(\cF)=CC(j^*\cF).
\end{equation*}
\item[(ii)]
 For a smooth $k$-scheme $P$ of equidimension $m$ and a closed immersion $h:X\rightarrow P$, we have
\begin{equation*}
h_\circ(CC(\cF))=CC(h_*\cF),
\end{equation*}
where $h_\circ(CC(\cF))$ denotes the push-forward of the cycle $(-1)^{m-n}dh^*(CC(\cF))$ on $\bT^*P\times_PX$ to $\bT^*P$.
\end{itemize}
\end{proposition}

The index formula for $\ell$-adic sheaves is the following:
\begin{theorem}[{\cite[Theorem 7.13]{Saito15-2}}]\label{indexformula}
Let $\bar k$ be an algebraic closure of $k$, $X$ a smooth and projective $k$-scheme and $\cF$ an object of $D^b_c(X,\Lambda)$. Then, we have
\begin{equation}\label{indexformulaeq}
\chi(X_{\bar k},\cF|_{X_{\bar k}})=\deg(CC(\cF),\bT^*_XX)_{\bT^*X},
\end{equation}
where $\chi(X_{\bar k},\cF|_{X_{\bar k}})$ denotes the Euler-Poincar\'e characteristic of $\cF|_{X_{\bar k}}$.
\end{theorem}

\begin{example}\label{curveCC}
Let $T$ be a smooth $k$-curve, $E$ an effective Cartier divisor of $T$, $T_0$ the complement of $E$ in $T$, $r:T_0\rightarrow T$ is the canonical injection and $\cG$ a locally constant and constructible sheaf of $\Lambda$-modules on $T_0$. Then, we have
\begin{equation*}
CC(r_!\cG)=-\mathrm{rk}_\Lambda(\cG)\cdot[\bT^*_TT]-\sum_{t\in E}\mathrm{dimtot}_t(\cG)\cdot[\bT^*_tT],
\end{equation*}
where $\mathrm{dimtot}_t(\cG)$ denotes the total dimension of $\cG$ at $t$ (\ref{defdimtot}). When $T$ is proper,  the index formula \eqref{indexformulaeq} for $r_!\cG$ is exactly the Grothendieck-Ogg-Shafarevich formula \cite{sga5}.
\end{example}

\subsection{}\label{tdvc}
In the following of this section, we assume that all schemes are over $\bZ[1/\ell]$. Let 
\begin{equation*}
\xymatrix{
D\ar[r]&X\ar[rr]^g\ar[rd]_f&&Y\ar[ld]^h\\
&&S
}
\end{equation*}
be a commutative diagram of morphisms of finite type of Noetherian schemes, where $g:X\rightarrow Y$ is smooth, $h:Y\rightarrow S$ is smooth of relative dimension $1$, $D$ is a closed subscheme of $X$ such that $f|_D:D\rightarrow S$ is qusai-finite. Denote by $U$ the complement of $D$ in $X$.
Let $\cF$ be an object of $D^b_c(X,\Lambda)$.
Assume that $f:X\rightarrow S$ is locally acyclic with respect to $\cF$ and that $g|_U:U\rightarrow Y$ is locally acyclic with respect to $\cF|_U$.
Let $x$ be a point of $D$ and $\bar x$ a geometric point of $X$ above $x$,  $\bar s$ a geometric point of $S$ above $s=f(x)$ and $\rR\Phi_{\bar x}(\cF|_{X_{\bar s}}, g_{\bar s})$ the stalk
of the vanishing cycles complex of $\cF|_{X_{\bar s}}$ relative to $g_{\bar s}:X_{\bar s}\rightarrow Y_{\bar s}$.
We define a function $\varphi_{\cF,g}:D\rightarrow \mathbb Z$ by
\begin{equation}\label{fundimtotvc}
\varphi_{\cF,g}(x)=\sum_i(-1)^i\dimtot(\rR^i\Phi_{\bar x}(\cF|_{X_{\bar s}}, g_{\bar s})),
\end{equation}
which is independent of the choice of geometric point $\bar x$ of $X$ above $x$. The following proposition is the semi-continuity of total dimensions of vanishing cycles complex. 

\begin{proposition}[{\cite[Proposition 2.16]{Saito15-2}}]\label{prop:sfunc}
We take the notation and assumptions of \ref{tdvc}. Then,
the function $\varphi_{\cF,g}:D\rightarrow \mathbb Z$ is constructible. If $f|_D:D\rightarrow S$ is \'etale, the function
\begin{equation*}
f_*(\varphi_{\cF,g}):S\rightarrow \mathbb Z,\ \ \ s\mapsto \sum_{x\in D_{\bar s}}\varphi_{\cF,g}(x)
\end{equation*}
is locally constant on $S$.
\end{proposition}

\subsection{}
Let $S$ be an Noetherian scheme, $f:X\rightarrow S$ a smooth morphism of finite type and $\cF$ an object of $D^b_c(X,\Lambda)$. A cycle $B=\sum_{i\in I} m_i[B_i]$ in $\mathbb T^*(X/S)$ is called the {\it characteristic cycle of} $\mathcal F$ {\it  relative to} $S$ if each $B_i$ is open and equidimensional over $S$ and if, for any algebraic geometric point $\bar s$ of $S$, we have 
\begin{equation}
B_{\bar s}=\sum_{i\in I} m_i[(B_i)_{\bar s}]=CC(\mathcal F|_{X_{\bar s}})
\end{equation}
We denote by $CC(\mathcal F, X/S)$ the characteristic cycle of $\mathcal F$ on $X$ relative to $S$. Notice that relative characteristic cycles do not always exist.

\begin{theorem}\label{maintheocc}
Let $S$ be an excellent Noetherian scheme, $f:X\rightarrow S$ a smooth morphism of finite type and $\cF$ an object of $D^b_c(X,\Lambda)$. Then, there exists a dominant and quasi-finite morphism $\pi:S'\rightarrow S$  such that the relative characteristic cycle $CC(\cF|_{X\times_SS'}, (X\times_SS')/S')$ exists.
\end{theorem}

Theorem \ref{maintheocc} generalizes part (ii) of Theorem \ref{introdellau} to higher relative dimensional cases.

\begin{proof}
We may assume that $S$ is affine and integral.
 By Proposition \ref{ccopenclose}, we are reduced to the case where $X$ is a projective space over $S$. We take the notation of  \ref{projspace} and \ref{legtran}, and we write $X=\bP$.  By Theorem \ref{sseqrssfiber}, we may assume that the relative singular support $C=SS(\mathcal F,\bP/S)$ exists and that, for any $s\in S$, we have $SS(\cF|_{\bP_s})=(C_s)_{\rm red}$.

Step 1. We denote by $\wt V$ the free $\cO_S$-module $\Gamma(\wt\bP,\cO_{\wt\bP}(1))$ and by $\wt V^{\vee}$ the free $\cO_S$-module $\Gamma(\wt\bP^\vee,\cO_{\wt\bP^\vee}(1))$. They are of rank $N$. Let $\wt\bG$ be the Grassmannian of projective line in $\wt\bP^\vee$, i.e., the Grassmannian $\mathrm{Gr}(\wt V^{\vee},2)$ of projective quotient $\cO_S$-module of $\wt V^{\vee}$ of rank $2$.
Let $\wt \bD$ be the universal line in $\wt\bP^\vee\times_S\wt \bG$, i.e., for each $s\in S$, 
 \begin{equation*}
\wt \bD_s=\left\{(x,y)\in \wt\bP^\vee_s\times_s\wt \bG_s\,\big{|}\,x\in L_y\right\},
\end{equation*}
 where $L_y\subseteq \wt\bP^\vee_s$ denotes the projective line associated to $y\in \wt \bG_s$. It is isomorphic to the Flag variety $\mathrm{Fl}(\wt V^{\vee}, 2,1)$ over $S$ characterized by surjections $\wt V^{\vee}\twoheadrightarrow L_2\twoheadrightarrow L_1$ of projective $\cO_S$-modules where $\mathrm{rank}_{\cO_S}(L_2)=2$ and $\mathrm{rank}_{\cO_S}(L_1)=1$.  Let $\wt \bA\subseteq\wt\bP\times_S \wt \bG$ be the universal axis, i.e., for any $s\in S$, 
\begin{equation*}
\wt \bA_s=\left\{(x,y)\in \wt\bP_s\times_s\wt \bG_s,\big{|}\,  x\in H_z,\, \textrm{for any}\,\,z\in L_y\right\},
\end{equation*}
where $H_z$ denotes the hyperplane of $\wt\bP_s$ associated to $z\in \wt\bP^\vee_s$. It is isomorphic to the Flag variety $\mathrm{Fl}(\wt V, N-2,1)$ over $S$ characterized by the surjections $\wt V\twoheadrightarrow L'_{N-1}\twoheadrightarrow L'_1$ of projective $\cO_S$-modules where $\mathrm{rank}_{\cO_S}(L'_{N-2})=N-2$ and $\mathrm{rank}_{\cO_S}(L'_1)=1$. Notice that $\wt\bA$ is of codimensional $2$ in $\wt\bP\times_S \wt \bG$. We denote by $\rho:\wt \bD\rightarrow\wt\bP^\vee$ and $\tau:\wt \bD\rightarrow\wt \bG$ the canonical projections and by $\pi:\bP\times_{\wt\bP}\wt\bH\rightarrow \wt\bP^{\vee}$ the composition of $\wt i\times\mathrm{id}:\bP\times_{\wt\bP}\wt\bH\rightarrow \wt\bH$ and the projection $\wt\rpr^\vee:\wt\bH\rightarrow\wt\bP^\vee$.
 We have a commutative diagram (cf. \cite[(5.2)]{Saito15-2})
\begin{equation}\label{eq:dia12}
\xymatrix{
\bP\times_{\wt\bP}\wt\bh\ar[d]_{\pi}\ar@{}|{\Box}[rd]&(\bP\times_S \wt\bG)'\ar[d]^{\pi'}\ar[l]_-(0.5){\rho'}\ar[r]^-(0.5){\tau'}&\bP\times_S \wt\bG\ar[d]^{\rpr_2}\\
\wt\bP^\vee&\wt\bD\ar[r]^{\tau}\ar[l]_{\rho}&\wt\bG
}
\end{equation}
where $(\bP\times_S\wt\bG)'$ denotes the blow-up of $\bP\times_S\wt\bG$ along $\wt\bA\bigcap(\bP\times_S\wt\bG)$ and $\tau'$ is the canonical projection. Due to \cite[Lemma 5.2]{Saito15-2}, the left square of \eqref{eq:dia12} is Cartesian. We put $(\bP\times_S \wt\bG)^\circ=(\bP\times_S \wt\bG)\setminus(\wt\bA\bigcap(\bP\times_S\wt\bG))$.

Step 2. In the rest of the proof, we simply put $C=SS(\cF,\bP/S)$. We denote by $\{C_\alpha\}_{\alpha\in I}$ the set of irreducible components of $C$, by $\wt C$ (resp. $\wt C_\alpha$) the inverse image $(d\wt i)^{-1}(C)$ (resp. $(d\wt i)^{-1}(C_\alpha)$) in $\bP \times_{\wt\bP}\ct(\wt\bp/S)$ (resp. $(d\wt i)^{-1}(C_\alpha)\subseteq \bP \times_{\wt\bP}\ct(\wt\bp/S)$), by $\rP(\wt C)$ (resp. $\rP(\wt C_\alpha)$) the projectivization of $\wt C$ (resp.  $\wt C_\alpha$) in $\bP\times_{\wt\bP}\wt\bH$ by Legendre transform and by $\wt D$ (resp. $\wt D_\alpha$) the image $\pi(\wt C)$ (resp. $\pi(\wt C_\alpha)$) in $\wt \bP^\vee$. By \cite[II, 6.9.1]{EGA4}, we may assume that, for any $\alpha\in I$,  $C_\alpha$ (resp. $\wt C_\alpha$ and $\wt D_\alpha$) is flat over $S$.

Let $\eta$ be the generic point of $S$. There exists a point $\eta'$ and finite morphism $\eta'\rightarrow\eta$ such that, for any $\alpha\in I$, the fiber product $C_{\eta'}=C\times_S\eta'$ is geometrically irreducible (hence, $\rP(\wt C_{\alpha})_{\eta'}$ and $(\wt D_\alpha)_{\eta'}$ are also geometrically irreducible). Let $S'$ be the normalization of $S$ in $\eta'$.  By Corollary \ref{cor:fex}, after replacing $S$ by a Zariski open dense subset of $S'$ (we still denote it by $S$ for simplicity), we may assume that the relative singular support $C$ exists and that, for any $\alpha\in I$, the generic fiber $(C_{\alpha})_\eta$ is geometrically integral (hence, $\bP(\wt C_{\alpha})_{\eta}$ and $(\wt D_\alpha)_\eta$ are also geometrically integral).  Replacing $S$ again by a Zariski open dense subscheme, we may assume that, for any $s\in S$ and any $\alpha\in I$, the fibers $(C_{\alpha})_s$, $(\wt D_\alpha)_s$ and $\rP(\wt C_\alpha)_s$ are geometrically integral (\cite[III, 9.7.7]{EGA4}), and that $(C_{\alpha})_s$'s are different from each other (\cite[III, 9.7.8]{EGA4}). Assume that the Veronese embedding $\wt i:\bP\rightarrow\wt\bP$ has degree $d\geq 3$. By Lemma \ref{fibradicial} and Proposition \ref{geomctod}, we may assume that, for any $s\in S$ and any $\alpha\in I$, the canonical projection $\pi_s:\rP(\wt C_\alpha)_s\rightarrow (\wt D_\alpha)_s$ is generically radicial.

Step 3. Let $(\bP\times_S\wt\bG)^\nabla$ be the largest open subset of $(\bP\times_S\wt\bG)^\circ$ such that the inverse image $Z(\widetilde{C})=\rP(\widetilde{C})\times_{\bP\times_{\wt\bP}\wt\bH}(\bP\times_S\wt\bG)^\nabla$ is quasi-finite over $\wt\bG$.  Using \cite[Lemma 3.10]{Saito15-2} fiberwisely, we obtain that the complement $(\bP\times_{\wt\bP}\wt\bH)\backslash \rP(\wt C)$ is the largest open subset of $\bP\times_{\wt\bP}\wt\bH$ where the test pair $(\pi,\rpr_1):\bP^{\vee}\leftarrow \bP\times_{\wt\bP}\wt\bH\rightarrow \bP$ is $C$-transversal relative to $S$. Then, $(\bP\times_S\wt\bG)^\nabla\setminus  Z(\wt C)$ is
the largest open subset of $(\bP\times_S\wt\bG)^\nabla$ where the test pair $(\pi\circ\rho',\rpr_1):\wt\bD\leftarrow (\bP\times_S\wt\bG)^\nabla\rightarrow\bP$ is $C$-transversal relative to $S$. Hence, $\pi': (\bP\times_S\wt\bG)^\nabla\setminus  Z(\wt C)\rightarrow \wt\bD$ is universally locally acyclic with respect to $\cG=\cF|_{(\bP\times_S\wt\bG)^\nabla}$. After shrinking $S$, we may assume that the canonical projection $\bP\rightarrow S$ is universally locally acyclic with to $\cF$, hence, that $\rpr_2:(\bP\times_S\wt\bG)^\nabla\rightarrow \wt \bG$ is universally locally acyclic with respect to $\cG$.

Consider the following commutative diagram
\begin{equation}\label{eq:sfunc}
\xymatrix{
 Z(\widetilde{C})\ar[r]&(\bP\times_S\wt\bG)^\nabla \ar[rd]_-(0.5){\rpr_2}\ar[rr]^-(0.5){\pi'}&&\wt\bD\ar[ld]^-(0.5){\tau}\\
&&\wt\bG}
\end{equation}
By {\cite[Proposition 2.16]{Saito15-2}} (cf. Proposition \ref{prop:sfunc}), the function \eqref{fundimtotvc}
\begin{equation*}
\varphi_{\cG,\pi'}:Z(\wt C)\rightarrow \mathbb Z,
\end{equation*}
is constructible. For any $\alpha\in I$, we put $Z(\wt C_\alpha)=\rP(\wt C_\alpha)\times_{\bP\times_{\wt\bP}\wt\bH}(\bP\times_S\wt\bG)^{\nabla}$.  Then, there exists an open dense subset $Z'(\wt C)$ of $Z(\wt C)$,  such that $\varphi_{\cG,\pi'}$ is locally constant on $Z'(\wt C)$ and that, for any $\alpha, \beta\in I$, the subsets $Z'(\wt C_{\alpha})=Z'(\wt C)\cap Z(\wt C_{\alpha})$ and $Z'(\wt C_{\beta})=Z'(\wt C)\cap Z(\wt C_{\beta})$ are disjoint. Notice that the projection $\rho':(\bP\times_S\wt\bG)^{\nabla}\rightarrow\bP\times_{\wt\bP}\wt\bH$ is smooth and has geometrically integral fibers. Hence, we may assume that, after shrinking $S$, for any $\alpha\in I$, the scheme $Z'(\wt C_{\alpha})$ is irreducible and, for any $s\in S$, the fiber $Z'(\wt C_{\alpha})_s$ is non-empty and geometrically irreducible.

Step 4. For any $\alpha\in I$, Let $\xi_\alpha$ (resp. $\zeta_\alpha$) be the generic point of $\rP(\wt C_\alpha)$  (resp. $\wt D_\alpha$) and $\varphi_\alpha$ the value of $\varphi_{\cG,\pi'}$ on $ Z'(\widetilde{C}_\alpha)$. We put
\begin{equation}\label{defA}
B=-\sum_{\alpha\in I}\frac{\varphi_\alpha}{[\xi_\alpha:\zeta_\alpha]}[C_\alpha].
\end{equation}
After replacing $S$ by a Zariski open dense subscheme, we may assume that, for any $\alpha\in I$ and any $s\in S$, the fiber $Z'(\wt C_\alpha)_s$ is non-empty. Let $\bar s$ be an algebraic geometric point of $S$. For any $\alpha\in I$, we denote by $(\xi_\alpha)_{\bar s}$ (resp. $(\zeta_\alpha)_{\bar s}$) the generic point of $\rP(\wt C_\alpha)_{\bar s}$ (resp. $(\wt D_\alpha)_{\bar s}$). By \cite[Proposition 5.8, Theorem 5.9]{Saito15-2}, the characteristic cycle $CC(\cF|_{\bP_{\bar s}})$ is defined by
\begin{equation}\label{defCCfiber}
CC(\cF|_{\bP_{\bar s}})=-\sum_{\alpha\in I}\frac{\varphi_{\alpha}}{[(\xi_{\alpha})_{\bar s}:(\zeta_{\alpha})_{\bar s}]}[(C_{\alpha})_{\bar s}].
\end{equation}
By Lemma \ref{lem:ab-4}, for any $\alpha\in I$ and any algebraic geometric point $\bar s$ of $S$, we have
\begin{equation}\label{eq:ab-4}
[\xi_\alpha:\zeta_\alpha]=[(\xi_{\alpha})_{\bar s}:(\zeta_{\alpha})_{\bar s}].
\end{equation}
By \eqref{defA}, \eqref{defCCfiber} and \eqref{eq:ab-4}, we obtain that, for any algebraic geometric point $\bar t$ of $S$, 
\begin{equation*}
B_{\bar t}=CC(\cF|_{\bP_{\bar t}}).
\end{equation*}
Hence, $B$ is the characteristic cycle of $\cF|_{X\times_SS'}$ relative to $S'$.
\end{proof}

\begin{proposition}[Saito]\label{flattransversal}
We assume that $k$ is algebraically closed. Let $T$ be a connected and smooth $k$-scheme of dimension $n$, $g:Y\rightarrow T$ a smooth morphism of finite type and $\cG$ an object of $D^b_c(Y,\Lambda)$. 
Assume that $g:Y\rightarrow T$ is $SS(\mathcal G)$-transversal and that each irreducible component of $SS(\mathcal G)$ is open and equidimensional over $T$. Then, the relative characteristic cycle $CC(\mathcal G, Y/S)$ exists, and we have 
\begin{equation}
CC(\mathcal G, Y/S)=(-1)^n\theta_*(CC(\mathcal G)),
\end{equation}
where $\theta:\mathbb T^*Y\rightarrow \mathbb T^*(Y/T)$ denotes the projection induced by the canonical map $\Omega^1_{Y/k}\rightarrow\Omega^1_{Y/T}$. 
\end{proposition}

\begin{proof}
 For any closed point $t$ of  $T$, we have the following Cartesian diagram
\begin{equation*}
  \xymatrix{\relax
 \mathbb T^*Y\ar[d]^{\theta}\ar@{}|-{\Box}[rd]&\mathbb T^*Y\times_YY_t\ar[l]_-(0.5){i'_t}\ar[d]^{\theta_t}\\
 \mathbb T^*(Y/T)&\mathbb T^*Y_t\ar[l]_-(0.5){i_t}}
\end{equation*}
where $i_t$ denotes the canonical injection. We may assume that $Y$ is of equidiemension $d+n$. Since $SS(\mathcal G)$ is equidimensional over $T$, for any closed point $t$ of $T$, the fiber product $SS(\mathcal G)\times_YY_t$ is of equidimension $d$. Since $g:Y\rightarrow T$ is $SS(\mathcal G)$-transversal, for any closed point $t$ of $T$, the canonical injection $\iota_t:Y_t\rightarrow Y$ is $SS(\mathcal G)$-transversal.  Hence, for any closed point $t\in T$, the canonical injection $\iota_t:Y_t\rightarrow Y$ is {\it propoerly} $SS(\mathcal G)$-{\it transversal} (\cite[Definition 7.1]{Saito15-2}). Then, for any closed point $t$ in $T$, we have (\cite[Theorem 7.6]{Saito15-2})
\begin{equation}\label{saito7.6}
CC(\mathcal G|_{Y_t})=(-1)^n\theta_{t*}(i'^*_t(CC(\mathcal G))).
\end{equation}
The restriction map $\theta:SS(\mathcal G)\rightarrow \theta(SS(\mathcal G))$ is finite, since $g:Y\rightarrow T$ is $SS(\mathcal G)$-transversal (\cite[Lemma 3.1]{Saito15-2}). Then, each irreducible component of $\theta(SS(\mathcal G))$ is also open and equidimensional over $T$. For any closed point $t$ of $T$, we have an equality of cycles
\begin{equation}\label{saito3.1}
  \theta_{t*}(i'^*_t(CC(\mathcal G)))=i^*_t(\theta_*(CC(\mathcal G))).
\end{equation}
on $\mathbb T^*Y_t$.  We put $B=(-1)^n\theta_*(CC(\mathcal G))$. By \eqref{saito7.6} and \eqref{saito3.1},
for any closed point $t$ of $T$, we have
\begin{equation}\label{closedfiberB}
 CC(\mathcal G|_{Y_t})=i^*_tB.
\end{equation}
By Theorem \ref{maintheocc},  there exists an integral $k$-scheme $T'$ and a dominant and quasi-finite morphism $\pi:T'\rightarrow T$ such that $CC(\mathcal G|_{Y\times_TT'},(Y\times_TT')/T')$ exists. By \eqref{closedfiberB} and the definition of relative characteristic cycles, we must have
\begin{equation}
\pi'^*B=CC(\mathcal G|_{Y\times_TT'},(Y\times_TT')/T'),
\end{equation}
where $\pi':\mathbb T^*(((Y\times_TT')/T')\rightarrow \mathbb T^*(Y/T)$ denotes the projection induced by $\pi:T'\rightarrow T$.
Hence, $B=CC(\mathcal G,Y/T)$. 
\end{proof}

\section{Failure of the lower semi-continuity for characteristic cycles}

\subsection{} In this section, let $k$ denotes an algebraically closed field of characteristic $p\geq 3$, $S=\mathrm{Spec}(k[u])$,  $X=\mathrm{Spec}(k[x,y,z])$, $D=(z)$, $U=X-D$ and $f:X\rightarrow S$ the projection associated to
\begin{equation*}
k[u]\rightarrow k[x,y,z],\ \ \ u\mapsto x.
\end{equation*}

\begin{example}\label{ssnotconst}
 Let $\cG$ be a locally constant and constructible sheaf of $\Lambda$-modules of rank $1$ on $U$ associated to the Artin-Schreier covering defined by $t^p-t=(xy)/z^p$ and we put $\cF=j_!\cG[2]$. For each $s\in S(k)$, we have $X_s=\mathrm{Spec}(k[y,z])$ and $U_s=\mathrm{Spec}(k[y,z,z^{-1}])$. The restriction $\cG|_{U_{s}}$ is a locally constant and constructible sheaf of $\Lambda$-modules of rank $1$ on $U_s$ associated to the Artin-Schreier covering defined by $t^p-t=(sy)/z^p$. Notice that $\cG|_{U_s}$ is the constant sheaf $\Lambda_{U_s}$ when $s=0$.

We have (\cite[Example 2.2]{Saito15-1})
\begin{equation*}
SS(\cF|_{X_s})=\left\{\begin{array}{ll}
\bT^*_{X_{s}}X_s \bigcup D_s\cdot\langle dz\rangle,&\text{if}\ \ s=0,\\
\bT^*_{X_s}X_s\bigcup D_s\cdot\langle dy\rangle,&\text{if}\ \ s\neq 0,
\end{array}
\right.
\end{equation*}
where $D_s\cdot\langle dz\rangle$ (resp. $D_s\cdot\langle dy\rangle$) denotes the sub-line bundle of $\bT^*X_s\times_{X_s}D_s$ spanned by the section $dz$ (resp. $dy$). We see that part (i) of Theorem \ref{introdellau} cannot be generalized to  $\cF$.

\end{example}
\begin{example}\label{ccnotconst}
We take the sheaf $\cG$ in Example \ref{ssnotconst}. Let $\cG_1$ be the constant sheaf $\Lambda_{U}$ and $\cG_2$ the locally constant sheaf of $\Lambda$-modules of rank $1$ on $U$ associated to the Artin-Schreier covering defined by $t^p-t=y/z^p$. We put $\cF'=j_!(\cG_1\oplus\cG_2\oplus\cG)[2]$. For any closed point $s\in S$, we have
\begin{equation*}
SS(\cF'|_{X_s})=\bT^*_{X_s}X_s\bigcup D_s\cdot\langle dy\rangle\bigcup D_s\cdot\langle dz\rangle.
\end{equation*}
We denote by $D\cdot\langle dy\rangle$ (resp. $D\cdot\langle dz\rangle$) the sub-line bundle of $\bT^*(X/S)\times_XD$ spanned by the section $dy$ (resp. $dz$). Hence, the closed subset
\begin{equation*}
A=\bT^*_{X}(X/S)\bigcup D\cdot\langle dy\rangle\bigcup D\cdot\langle dz\rangle
\end{equation*}
of $\bT^*(X/S)$ satisfies part (i) of Theorem \ref{introdellau} for $\cF'$ in the higher relative dimension situation. However, by \cite[Example 3.6]{Saito13}, we have
\begin{equation*}
CC(\cF'|_{X_s})=\left\{\begin{array}{ll}
3[\bT^*_{X_s}X_s]+2[D_s\cdot\langle dz\rangle]+p[D_s\cdot\langle dy\rangle],&\text{if}\ \ s=0,\\
3[\bT^*_{X_s}X_s]+[D_s\cdot\langle dz\rangle]+2p[D_s\cdot\langle dy\rangle],&\text{if}\ \ s\neq0.
\end{array}\right.
\end{equation*}
It implies that $\cF'$ does not have a lower semi-continuous property at the origin of $S$. Hence, we cannot find a finite and surjective map $f:S'\rightarrow S$ and a cycle $B$ supported on $A\times_SS'$ that generalize part (iii) of Theorem \ref{introdellau} for $\cF'$. 
\end{example}

\end{document}